\documentclass[11pt]{amsart}
\usepackage[utf8]{inputenc}
\usepackage[T1]{fontenc}
\usepackage{lmodern}
\usepackage{amsmath,amssymb,amsthm,mathtools}
\usepackage[margin=1in]{geometry}
\usepackage[expansion=false]{microtype}
\usepackage{enumitem}
\usepackage{graphicx}
\usepackage{xcolor}
\usepackage[colorlinks=true,allcolors=black]{hyperref}

\newcommand{\T}{\mathbb{T}}
\newcommand{\R}{\mathbb{R}}

\newcommand{\Z}{\mathbb{Z}}
\newcommand{\N}{\mathbb{N}}
\newcommand{\defeq}{\coloneqq}
\DeclareMathOperator{\supp}{supp}
\DeclareMathOperator{\sgn}{sgn}
\DeclareMathOperator{\id}{id}
\DeclareMathOperator{\dist}{dist}

\theoremstyle{plain}
\newtheorem{theorem}{Theorem}
\newtheorem{lemma}{Lemma}[section]
\newtheorem{proposition}[lemma]{Proposition}
\theoremstyle{definition}
\newtheorem{definition}[lemma]{Definition}
\theoremstyle{remark}
\newtheorem{remark}{Remark}
\newtheorem{corollary}{Corollary}

\numberwithin{equation}{section}
\setlist[enumerate]{leftmargin=2em,label=\textup{(\roman*)},itemsep=3pt,topsep=5pt}
\allowdisplaybreaks[1]
\title[On $C^2$ circle diffeomorphisms with rotation numbers of constant type]
{On $C^2$ circle diffeomorphisms with rotation numbers of constant type}
\author{Junhao Li}
\address{School of Mathematics and Statistics, Beijing Institute of Technology, Beijing 100081, China}
\email{junhaoli@bit.edu.cn}
\author{Lin Wang}
\address{School of Mathematics and Statistics, Beijing Institute of Technology, Beijing 100081, China}
\email{lwang@bit.edu.cn}
\subjclass[2020]{Primary 37E10; Secondary 37C15, 42A45, 46E35}
\keywords{Circle diffeomorphisms, rotation numbers of constant type, Besov spaces, weak Dini continuity}

\begin{document}

\begin{abstract}
For every irrational rotation number $\alpha$ of constant type and every $q>1$, we construct a $C^2$ circle diffeomorphism $F$ arbitrarily close to the rigid rotation $R_\alpha$ in the combined $C^2$ and $B^2_{2,q}$ topology whose normalized conjugacy is not Lipschitz. We also prove that $F''\in B^0_{2,1}$ implies global $C^1$ conjugacy. Thus the thresholds in the Besov summability and weak Dini scales are $q=1$ and $\gamma=1/2$, respectively. These results complement Herman's Baire category theorem on the residual existence of $C^2$ diffeomorphisms without $C^1$ conjugacy by providing quantitative counterexamples and sufficient conditions for differentiable conjugacy.
\end{abstract}
\maketitle

\setcounter{tocdepth}{2}
\begingroup
\fontsize{9}{11}\selectfont
\setlength{\linespacing}{11pt}
\color{black}
\hypersetup{linkcolor=black,linktoc=all}
\tableofcontents
\endgroup
\hypersetup{linkcolor=black}
\clearpage

\section{Introduction}\label{sec:introduction}

Let $\T\defeq\R/\Z$ be equipped with normalized Lebesgue measure. For an integer $r\ge1$, let $\operatorname{Diff}_+^r(\T)$ denote the group of orientation-preserving $C^r$ circle diffeomorphisms, and let
\[
\mathcal D^r(\T)\defeq\{F\in\operatorname{Diff}_+^r(\R):F(x+1)=F(x)+1\}
\]
denote the space of their orientation-preserving lifts to the real line. Write $\id(x)=x$. For $F \in \mathcal{D}^r(\T)$, denote its inverse by $F^{-1}$ and its $n$th iterate by $F^n$, with $F^0 = \id$. The rotation number of $F$ is
\[
\rho(F)\defeq\lim_{n\to\infty}\frac{F^n(x)-x}{n}\in\R,
\]
where the limit is independent of $x$. Write $R_\alpha(x)\defeq x+\alpha$ for the rigid rotation, and denote the class of lifts with fixed rotation number $\alpha$ by
\[
\mathcal F_\alpha^r\defeq\{F\in\mathcal D^r(\T):\rho(F)=\alpha\}.
\]
We also use $\mathcal D^0(\T)$ for the space of increasing homeomorphisms $F:\R\to\R$ satisfying $F(x+1)=F(x)+1$. Superscripts $\infty$ and $\omega$ denote smooth and real-analytic regularity, respectively, and noninteger regularity exponents are interpreted in the H\"older sense. The definition of $\mathcal F_\alpha^r$ extends to these classes by the same rotation-number condition.

The induced rotation on $\T$ is also denoted by $R_\alpha$, and we write $\|x\|\defeq\dist(x,\Z)$.

\begin{definition}\label{def:constant-type}
If there exists a constant $c>0$ such that an irrational number $\alpha$ satisfies
\begin{equation}\label{eq:certificate}
\|k\alpha\|\ge c/|k|\qquad(k\in\Z\setminus\{0\}),
\end{equation}
then $\alpha$ is said to be of \emph{constant type}.
\end{definition}
Fix $\alpha$ and a constant $c$ satisfying~\eqref{eq:certificate}. We call $c$ a \emph{Markov bound} for $\alpha$. Hurwitz's theorem gives $c\le1/\sqrt5<1/2$; see \cite{Sch80}.

\subsection{Background and motivation}\label{subsec:history}

The conjugacy theory of circle homeomorphisms originates with Poincar\'e and Denjoy. For every irrational rotation number $\alpha$, Denjoy's theorem gives a topological conjugacy between each $F\in\mathcal F_\alpha^2$ and $R_\alpha$~\cite{Den32}. Denjoy also constructed $C^1$ circle diffeomorphisms with irrational rotation number that are not topologically conjugate to a rotation~\cite[Proposition 12.2.1]{KH95}. The differentiability of the conjugacy depends on the regularity of the map and the arithmetic properties of the rotation number.

For rotation numbers of constant type, Katznelson and Ornstein~\cite{KO89} proved that the conjugacy associated with every $F\in\mathcal F_\alpha^2$ is absolutely continuous. Herman asked whether every such map admits a $C^1$ conjugacy to the rotation~\cite[p.~161, Problem 4.4]{Her79}. Using an idea of Meyer~\cite{Mey83}, Herman applied the Baire category theorem to obtain a dense residual set of maps in $\mathcal F_\alpha^2$ admitting no $C^1$ conjugacy to $R_\alpha$; these notes were published in~\cite{Her18}.

Sufficient conditions for differentiable conjugacy have also been formulated in terms of H\"older, Besov, and Dini regularity. Herman~\cite[Proposition VI.5.1]{Her86} established local $C^{1+\varepsilon}$ conjugacy under a smallness condition in $B^{2+\varepsilon}_{p,\infty}$, for $p>1$ and $0<\varepsilon<1$, with the smallness threshold depending on the Markov bound and the regularity parameters. Khanin and Teplinsky~\cite{KT09} proved global $C^{1+\varepsilon}$ conjugacy for $C^{2+\varepsilon}$ maps, where $0<\varepsilon<1$. In this range, $C^{2+\varepsilon}=B^{2+\varepsilon}_{\infty,\infty}$ with equivalent norms. Tong, Xiao, and Li~\cite{TXL} proved global $C^1$ conjugacy when $F''$ satisfies the classical Dini condition and obtained estimates for the modulus of continuity of the derivative of the conjugacy.

We study sufficient conditions and counterexamples in terms of the modulus of continuity of $F''$ and the summability of its dyadic Fourier blocks.

\smallskip
\noindent\textbf{Question.} {\it Which conditions in these two scales ensure global $C^1$ conjugacy for $C^2$ circle diffeomorphisms with rotation numbers of constant type?}
\smallskip

Our results give a global sufficient condition and identify the thresholds in these two scales through counterexamples. We first introduce the function spaces and metrics used in their statements.

\subsection{Function spaces and metrics}\label{subsec:spaces-metrics}

Throughout this paper, we set $\N\defeq\{1,2,\ldots\}$ and $\N_0\defeq\N\cup\{0\}$. All function norms are understood to be defined on $\T$. For an integer $r\ge0$, write $D \defeq d/dx$ and define the $C^r$ norm by $\|u\|_{C^r} \defeq \max_{0\le j\le r}\|D^j u\|_\infty$; for a noninteger exponent $r+\sigma$, where $0<\sigma<1$, the notation $C^{r+\sigma}$ denotes the classical H\"older space. For iterates, we write $DF^n \defeq D(F^n)$. We abbreviate the Lebesgue norm as $\|w\|_p \defeq \|w\|_{L^p(\T)}$ and equip the periodic Sobolev space $W^{1,p}(\T)$ with the norm $\|w\|_{W^{1,p}} \defeq \|w\|_p + \|Dw\|_p$.

Unless explicitly stated to be absolute, $C$ denotes a generic positive constant depending on the Markov bound and the choice of norms, and may vary from one occurrence to the next. We write $X=O(Y)$ if $|X|\le CY$, and $X\asymp Y$ if both $X=O(Y)$ and $Y=O(X)$ hold. The notation $X\hookrightarrow Y$ denotes a continuous embedding.

Let $\mathcal{D}'(\T)$ denote the space of periodic distributions, that is, the dual of the space $C^\infty(\T)$ of smooth test functions, with duality pairing denoted by $\langle w,\phi\rangle$. For $w\in L^1(\T)$, we use the convention $\langle w,\phi\rangle \defeq \int_\T w\phi$. The weak derivative of a distribution is defined by $\langle Dw,\phi\rangle \defeq -\langle w,\phi'\rangle$. For $w\in\mathcal{D}'(\T)$, its Fourier coefficients and sharp dyadic frequency projections are given by
\[
\widehat w(k)\defeq \langle w, e^{-2\pi ikx}\rangle,\qquad \Pi_jw\defeq \sum_{2^j\le |k|<2^{j+1}}\widehat w(k)e^{2\pi ikx}\quad (j\ge0).
\]

\begin{definition}[Besov spaces and frequency-domain functionals]\label{def:besov}
Let $s\in\R$ and $p,q\in[1,\infty]$. Choose an even cutoff function $\chi\in C_c^\infty(\R)$ such that $0\le\chi\le1$, equal to 1 on $[-1,1]$ and to 0 outside $[-2,2]$. For $w\in\mathcal{D}'(\T)$, define the smooth dyadic projections by
\begin{equation}\label{eq:smooth-blocks}
\widehat{\mathsf P_0w}(k)\defeq \chi(k)\widehat w(k),\qquad
\widehat{\mathsf P_jw}(k)\defeq [\chi(2^{-j}k)-\chi(2^{1-j}k)]\widehat w(k)\quad(j\ge1).
\end{equation}
The sequence $(\mathsf P_j)_{j\ge0}$ forms a resolution of the identity in the sense of distributions, namely $w=\sum_{j\ge0}\mathsf P_jw$. The periodic Besov space $B^s_{p,q}(\T)$ is equipped with the norm
\begin{equation}\label{eq:general-besov-norm}
\|w\|_{B^s_{p,q}(\T)}\defeq \big\|\big(2^{sj}\|\mathsf P_jw\|_{L^p(\T)}\big)_{j\ge0}\big\|_{\ell^q}.
\end{equation}
To quantify critical regularity, define
\[
\mathcal B_q(w)\defeq |\widehat w(0)| + \big\|\big(\|\Pi_jw\|_2\big)_{j\ge0}\big\|_{\ell^q}.
\]
If $\mathcal B_q(w)<\infty$, we say that $w$ satisfies the frequency $L^2$-weak Dini condition with exponent $q$.
\end{definition}

\begin{definition}[Weak Dini conditions]\label{def:weak-dini}
For a continuous function $w$, define its modulus of continuity and weak Dini functional by
\[
\omega_w(t)\defeq \sup_{\substack{x\in\T\\ |u|\le t}}|w(x+u)-w(x)|,\qquad
\mathcal D_\gamma(w)\defeq \int_0^{e^{-1}} \frac{\omega_w(t)}{t(\ln(1/t))^\gamma}\,dt
\quad(0\le\gamma\le1).
\]
If $\mathcal D_\gamma(w)<\infty$, we say that $w$ satisfies the $\gamma$-weak Dini condition; the case $\gamma=0$ is the classical Dini condition.
\end{definition}

The functional $\mathcal B_q$ is uniformly equivalent to the $B^0_{2,q}(\T)$ norm; see Proposition~\ref{appB:block-calculus} for a proof. All Besov spaces in this paper are defined using the smooth resolution of the identity $(\mathsf P_j)_{j\ge0}$, so as to include the endpoints $p\in\{1,\infty\}$. The sharp dyadic projections $\Pi_j$ are used specifically to define $\mathcal B_q$ through $L^2$ orthogonality and to derive frequency-domain energy estimates. For functions of zero mean, $\mathcal B_q$ is precisely the $\ell^q$ norm of the sequence of $L^2$ norms of the dyadic blocks. By the inclusions in~\eqref{appB:q-monotonicity}, for $1\le q_1\le q_2$ we have the continuous embeddings
\[
B^0_{2,1}(\T)\subset B^0_{2,q_1}(\T)\subset B^0_{2,q_2}(\T).
\]

Membership of lifts of circle homeomorphisms in Sobolev and Besov spaces is defined through their periodic displacements. Specifically, for any lift $h$, its derivative $h'$ is a well-defined periodic function, and the assertion $h\in W^{1,p}(\T)$ means $h-\id\in W^{1,p}(\T)$. Similarly, the differentiation property of Besov spaces~\eqref{appB:lifting} gives $F''\in B^0_{2,q}(\T)\iff F-R_\alpha\in B^2_{2,q}(\T)$. On the space of lifts $\mathcal{D}^2(\T)$, define
\begin{equation}\label{eq:d2-definition}
d_2(F,G)\defeq \max_{0\le j\le2}\|D^j(F-G)\|_\infty,
\qquad
\mathfrak d_q(F,G)\defeq d_2(F,G)+\mathcal B_q(F''-G'').
\end{equation}
Here $d_2$ is a metric on $\mathcal D^2(\T)$, whereas $\mathfrak d_q$ is an extended metric. Its restriction to the subclass of maps satisfying $F''\in B^0_{2,q}(\T)$ is a finite-valued metric.

\subsection{Main results}\label{subsec:main-results}

With the function spaces and metrics defined above, our main results are as follows.

\begin{theorem}[Counterexamples up to the Besov endpoint]\label{thm:main}
Let $\alpha$ be any irrational number of constant type. For every $q_*>1$ and $\varepsilon>0$, there exists $F\in\mathcal F_\alpha^2$ such that
\begin{gather*}
\mathfrak d_{q_*}(F,R_\alpha)<\varepsilon,\qquad
F''\in\bigcap_{q>1}B^0_{2,q}(\T),\qquad F''\notin B^0_{2,1}(\T).
\end{gather*}
Let $h$ be the corresponding normalized increasing conjugacy, that is,
\[
 F\circ h=h\circ R_\alpha,\qquad h(0)=0,\qquad h(x+1)=h(x)+1.
\]
Then $h'>0$ almost everywhere, and
\begin{align*}
 h'\in\bigcap_{1<p<\infty}B^0_{p,2}(\T),\quad
 h-\id,\ h^{-1}-\id\in\bigcap_{1<p<\infty}W^{1,p}(\T),\quad
 h'\notin L^\infty(\T).
\end{align*}
In particular, $F$ and $R_\alpha$ are not $C^1$ conjugate. Moreover, $F$ can be chosen so that $\mathcal D_\gamma(F'')<\infty$ for every $1/2<\gamma\le1$.
\end{theorem}

\begin{corollary}[Density and category of the counterexamples]
\label{cor:dense-meagre}
Fix a constant-type irrational number $\alpha$, and give
$\mathcal F_\alpha^2$ the relative $C^2$ topology. For
$F\in\mathcal F_\alpha^2$, let $h_F$ be its normalized increasing
conjugacy, so that
\[
 F\circ h_F=h_F\circ R_\alpha,\qquad
 h_F(0)=0,\qquad h_F(x+1)=h_F(x)+1.
\]
Let $\mathcal E_\alpha$ be the set of all such $F$ satisfying
\[
 F''\in\left(\bigcap_{q>1}B^0_{2,q}(\T)\right)
       \setminus B^0_{2,1}(\T),\qquad
 \mathcal D_\gamma(F'')<\infty\quad
       \text{for every }\tfrac12<\gamma\le1,\qquad
 h_F'\notin L^\infty(\T).
\]
Then $\mathcal E_\alpha$ is dense and meagre in
$\mathcal F_\alpha^2$.
\end{corollary}

\begin{theorem}[Global linearization at the Besov endpoint]\label{thm:rigidity}
Let $\alpha$ be an irrational number of constant type, and let $F\in\mathcal F_\alpha^2$. If
\[
 F''\in B^0_{2,1}(\T),
\]
then there exists an orientation-preserving diffeomorphism $h\in\mathcal D^1(\T)$ such that $F\circ h=h\circ R_\alpha$. Equivalently, the condition $F-R_\alpha\in B^2_{2,1}(\T)$ implies $C^1$ linearization.
\end{theorem}

\begin{remark}[Sharpness of the summability exponent]\label{rem:sharpness-q}
In the class $\mathcal F_\alpha^2$, the summability exponent $q=1$ in the implication
\[
F'' \in B^0_{2,q}(\T) \implies F \text{ and } R_\alpha \text{ are } C^1 \text{ conjugate},
\]
is optimal: for $q=1$, Theorem~\ref{thm:rigidity} applies to every map in this class, whereas for every $q>1$, Theorem~\ref{thm:main} provides counterexamples in every $\mathfrak d_q$ neighborhood of $R_\alpha$.
\end{remark}

\begin{remark}[Regularity of the conjugacy]\label{rem:regularity-h}
The assertions concerning Sobolev and Besov regularity of the conjugacy in Theorem~\ref{thm:main} hold for every map in $\mathcal F_\alpha^2$. Proposition~\ref{prop:conjugacy-regularity} derives these conclusions from the interval distortion estimates of Katznelson and Ornstein \cite{KO89}. The key point of the counterexample construction is to force $h'\notin L^\infty(\T)$, thereby destroying Lipschitz regularity. Herman~\cite[\S2.2]{Her98} also discussed regularity questions in Sobolev topologies.
\end{remark}

\begin{remark}[Sharpness in the weak Dini scale]\label{rem:sharpness-dini}
By Lemma~\ref{lem:half-dini} and Theorem~\ref{thm:rigidity}, for every $F\in\mathcal F_\alpha^2$ we have the chain of implications
\begin{equation}\label{eq:implication-chain}
\mathcal D_{1/2}(F'')<\infty \implies F''\in B^0_{2,1}(\T) \implies F \text{ and } R_\alpha \text{ are } C^1 \text{ conjugate}.
\end{equation}
The first implication is strict even within a class with fixed rotation number; see Remark~\ref{rem:half-dini-strict}. Since $\mathcal D_{1/2}(F'')\le \mathcal D_\gamma(F'')$ for $0\le\gamma\le 1/2$, the weak Dini conditions in this range are globally sufficient, extending the classical Dini result of \cite[Corollary 4.1]{TXL}. Conversely, for every fixed $1/2<\gamma_*\le 1$, Theorem~\ref{thm:main} provides maps satisfying $d_2(F, R_\alpha)+\mathcal D_{\gamma_*}(F'')<\varepsilon$ that are not $C^1$ conjugate to the rotation. Thus $\gamma=1/2$ is the  sharp threshold.
\end{remark}

\begin{remark}[Necessity of the constant type assumption]\label{rem:necessity-constant-type}
The constant type assumption is indispensable for absolute continuity of the conjugacy for general $C^2$ maps. For every irrational number $\alpha$ that is not of constant type, Hawkins and Schmidt \cite{HS82} constructed a $C^2$ circle diffeomorphism whose induced transformation is of Krieger type $\mathrm{III}_1$. Such maps admit no absolutely continuous conjugacy to $R_\alpha$, and hence no $C^1$ conjugacy.
\end{remark}

\subsection{Outline of the proofs and organization of the paper}\label{subsec:strategy}

The proof of Theorem~\ref{thm:main} uses a Fej\'er-regularized small-divisor kernel whose $L^1$ norm diverges at least as fast as $\sqrt{\ln N}$. Taking an antiderivative of the sign of the kernel produces a normalized potential with supremum 1 and a tightly controlled cohomological difference. Exponentiating this potential yields an analytic conjugacy whose derivative undergoes strong local oscillations, while the conjugated circle map remains sufficiently close to the rigid rotation in the $C^2$ topology. The amplitude determines the scale of derivative growth, while the Fourier cutoff order is used to control the $C^2$, Besov, and weak Dini norms.

Each stage of the inductive construction combines such an analytic perturbation with the insertion of a new frequency block. Throughout the iteration, three nondegeneracy bounds are maintained uniformly: a positive lower bound for the derivative of a finite iterate, lower bounds for selected secant slopes of the normalized conjugacy, and a lower bound for the local $B^0_{2,1}$ modulus of the second derivative. Summable error bounds ensure convergence to a $C^2$ limit diffeomorphism that inherits these inequalities. After each insertion, Yoccoz's analytic linearization theorem \cite{Yoc84} is used to update the background conjugacy for the next stage; see Figure~\ref{fig:recursion} for the order of parameter selection. The Sobolev and Besov estimates for the limiting conjugacy rely directly on the square-summability and interval ratio estimates in Propositions~\ref{appC:KO-square-summability} and \ref{appC:KO-interval-ratios}; their derivation is given in Section~\ref{subsec:conjugacy-regularity}.

For the density and category assertions in Corollary~\ref{cor:dense-meagre}, we first establish that the critical regularity conditions defining $\mathcal E_\alpha$ are preserved under smooth conjugation by exploiting the quasi-diagonal decay of sharp dyadic projections under changes of coordinates. Approximating any $C^2$ map by real-analytic lifts, linearizing them via Yoccoz's theorem, and applying Theorem~\ref{thm:main} yields the density of $\mathcal E_\alpha$. Conversely, by introducing high-frequency sinusoidal perturbations to blow up the weak Dini integrals locally, we place $\mathcal E_\alpha$ inside a countable union of closed nowhere dense sets, establishing its meagreness in $\mathcal F_\alpha^2$.

In the proof of Theorem~\ref{thm:rigidity}, Proposition~\ref{appC:KO-interval-ratios} establishes absolute continuity of the normalized conjugacy and $h'\in L^2(\T)$, which yield measure estimates for long intervals in the dynamical partitions. The argument is divided at the characteristic block frequency: below this frequency, integral oscillations of periodic primitives control the block averages; above it, the fundamental theorem of calculus controls the averaging error. Passing from $L^2$ to $L^{3/2}$ yields decay by a positive power of the measure of the exceptional set, making these estimates absolutely summable over the partition scales. The $B^0_{2,1}$ assumption on the second derivative ensures absolute summability of the averaging increments.
Summing the recursion in Proposition~\ref{appC:KO-recursion} and applying the comparison estimate in Proposition~\ref{appC:KO-smallness} gives absolute summability of the return distortions. Together with the interval ratio estimate in Proposition~\ref{appC:KO-interval-ratios}, this ensures that the dyadic difference quotients converge uniformly to an everywhere positive continuous limit $h'$, yielding $h\in\mathcal D^1(\T)$ by a direct argument.

The remainder of the paper is organized as follows. Sections~2--4 establish the kernel estimates, construct the potentials, and prove the analytic perturbation bounds. Section~\ref{sec:construction} completes the inductive construction and proves Theorem~\ref{thm:main}. Section~\ref{sec:density-category} proves Corollary~\ref{cor:dense-meagre}. Section~\ref{sec:rigidity} proves Theorem~\ref{thm:rigidity}. Appendix~\ref{app:square} contains the finite-band square estimate. Appendix~\ref{app:function-spaces} lists the function-space properties used for $h'$ and periodic displacements. Appendix~\ref{app:KO} collects the dynamical results of Katznelson and Ornstein used in Section~\ref{subsec:conjugacy-regularity} and Section~\ref{sec:rigidity}.

\subsection*{Acknowledgments}
This work was partially supported by the National Natural Science Foundation of China (Grant Nos.~12671217 and 12122109). The authors sincerely thank Dr.~Zhicheng Tong for his valuable suggestions, in particular his suggestion to study the modulus of continuity of $F''$ in Theorem~\ref{thm:main}.

\section{Fourier kernels and quantitative norm estimates}\label{sec:fourier}

We first estimate the regularized kernel of the cohomological equation. Bounds on its dyadic frequency blocks yield the growth estimate for the $L^1$ norm required in the construction. For basic properties of Fourier series and Fej\'er kernels, see \cite{SS03}. Write
\[
e_k(x)\defeq e^{2\pi ikx},\qquad
\Delta_\alpha u\defeq u\circ R_\alpha-u,\qquad
A_\alpha\defeq D\Delta_\alpha.
\]
Since $A_\alpha e_k=2\pi ik(e^{2\pi ik\alpha}-1)e_k$, define
\begin{equation}\label{eq:multiplier}
m_0\defeq0,\qquad
m_k\defeq\frac{1}{2\pi ik(e^{2\pi ik\alpha}-1)}\quad(k\ne0).
\end{equation}
Using $|e^{2\pi it}-1|=2\sin(\pi\|t\|)$, we have
\[
4\|k\alpha\|\le|e^{2\pi ik\alpha}-1|\le2\pi\|k\alpha\|,
\]
and hence \eqref{eq:certificate} gives, for every $k\ne0$,
\begin{equation}\label{eq:multiplier-bounds}
\frac{1}{4\pi^2(|k|\|k\alpha\|)}
\le |m_k|\le\frac{1}{8\pi(|k|\|k\alpha\|)}
\le\beta_c,\qquad
\beta_c\defeq\frac{1}{8\pi c}.
\end{equation}

\begin{remark}\label{rem:multiplier}
For the cohomological equation $\Delta_\alpha u=\varphi$, where $u$ and $\varphi$ have zero mean, the Fourier convention in the introduction gives
\[
\widehat u(k)
=\frac{\widehat\varphi(k)}{e^{2\pi ik\alpha}-1}
=m_k\widehat{D\varphi}(k)\qquad(k\ne0).
\]
Since the sequence $(m_k)$ is bounded, it defines a periodic distribution $\mu_\alpha$ by $\widehat{\mu_\alpha}(k)\defeq m_k$: explicitly, $\langle\mu_\alpha,\phi\rangle=\sum_k m_k\widehat\phi(-k)$, and the sum converges absolutely for every smooth periodic function $\phi$. The zero-mean solution is represented in the sense of distributions by $u=\mu_\alpha*D\varphi$. The notation $\mu_\alpha\sim\sum_{k\ne0}m_ke_k$ denotes a formal Fourier series and does not assert pointwise convergence. Here convolution means pointwise multiplication of the Fourier coefficients. The normalization is specified by $m_0=0$. For the difference operator with the opposite sign, $u-u\circ R_\alpha$, the multiplier is $-m_k$; see \cite{Mey83,Her18}.
\end{remark}

For $N\in\N$, define the Fej\'er weights and the regularized kernel by
\begin{align}
\sigma_N(k)&\defeq\max\left\{0,1-\frac{|k|}{N+1}\right\}
\qquad(k\in\Z),\label{eq:weights}\\
P_N(x)&\defeq\sum_{0<|k|\le N}\sigma_N(k)m_ke_k(x),
\qquad L_N\defeq\|P_N\|_1.\label{eq:kernel}
\end{align}
Since $m_{-k}=\overline{m_k}$, the polynomial $P_N$ is real-valued and has zero mean. Its coefficient at $k=1$ is nonzero for every $N\ge1$, so $L_N>0$.

\subsection{Dyadic coefficient estimates}\label{subsec:dyadic}
For $s\in\N_0$, write
\[
I_s\defeq\{k\in\Z:2^s\le|k|<2^{s+1}\},\qquad K_s\defeq2^s.
\]
For fixed $N$, write
\[
a_k\defeq\sigma_N(k)m_k,\qquad
n_s(\tau)\defeq\#\{k\in I_s:|a_k|>\tau\}\quad(\tau>0).
\]
If $k,l\in I_s$ are distinct, then $|k-l|<4K_s$, and \eqref{eq:certificate} gives
\[
\|(k-l)\alpha\|>\frac{c}{4K_s}.
\]
If $|a_k|>\tau$, then \eqref{eq:multiplier-bounds} gives $\|k\alpha\|<(8\pi K_s\tau)^{-1}$. For $\tau\ge(4\pi K_s)^{-1}$, the points $k\alpha\in\T$ lie on an arc of length $(4\pi K_s\tau)^{-1}\le1$. The separation estimate yields
\[
n_s(\tau)\le1+\frac{1}{\pi c\tau}.
\]
Since $|a_k|\le\beta_c$, $n_s(\tau)$ vanishes for $\tau\ge\beta_c$. For $0<\tau\le\beta_c$, we clearly have $1\le(8\pi c\tau)^{-1}$; for $0<\tau<(4\pi K_s)^{-1}$, the relations $\#I_s=2K_s$ and $c\le1/2$ give
\[
n_s(\tau)\le \#I_s=2K_s < (2\pi\tau)^{-1}\le(4\pi c\tau)^{-1}.
\]
Combining the two estimates gives, for every $\tau>0$, the uniform upper bound
\begin{equation}\label{eq:distribution}
n_s(\tau)\le\frac{5}{4\pi c\tau}\qquad(\tau>0).
\end{equation}
Define the constant
\[
S_c\defeq\frac{5}{2\pi^{4/3}c^{4/3}}.
\]
By \eqref{eq:distribution}, at the exponent $p=4/3$ we obtain
\begin{equation}\label{eq:block-sum}
\sum_{k\in I_s}|a_k|^{4/3}
=\frac43\int_0^{\beta_c}\tau^{1/3}n_s(\tau)\,d\tau
\le\frac43\int_0^{\beta_c}\frac{5\tau^{-2/3}}{4\pi c}\,d\tau
=S_c.
\end{equation}
Write
\[
W_s(x)\defeq\sum_{k\in I_s}a_ke_k(x),\qquad
\mathfrak n_N\defeq1+\lfloor\log_2 N\rfloor.
\]
Then $P_N=\sum_{s=0}^{\mathfrak n_N-1}W_s$.

\subsection{Upper bounds for the \texorpdfstring{$L^4$}{L4} norm}\label{subsec:L4-bound}

For a fixed frequency block $I_s$, let $b_k \defeq a_k\mathbf 1_{I_s}(k)$ be the truncated sequence. We use a consequence of the Cauchy--Schwarz inequality. For the compactly supported sequence $b$, set $B \defeq \sum_{k\in\Z} |b_k|^{4/3}$. For every $n\in\Z$, we have
\[
|(b*b)_n|^2 \le \bigg(\sum_{k\in\Z}|b_k|^{4/3}|b_{n-k}|^{4/3}\bigg)\bigg(\sum_{k\in\Z}|b_k|^{2/3}|b_{n-k}|^{2/3}\bigg)
\le B \sum_{k\in\Z}|b_k|^{4/3}|b_{n-k}|^{4/3},
\]
so that
\[
\|b*b\|_{\ell^2}^2 = \sum_{n\in\Z}|(b*b)_n|^2 \le B\sum_{k\in\Z}|b_k|^{4/3}\sum_{n\in\Z}|b_{n-k}|^{4/3} = B^3.
\]
Together with Plancherel's identity, this gives
\begin{equation}\label{eq:block-L4}
\|W_s\|_4^4 = \|W_s^2\|_2^2 = \|b * b\|_{\ell^2}^2 \le B^3 = \bigg(\sum_{k\in I_s} |a_k|^{4/3}\bigg)^3 \le S_c^3.
\end{equation}
To sum these dyadic frequency blocks, we use the following finite-band estimate; an elementary proof is given in Appendix~\ref{app:square}.

\begin{lemma}\label{lem:finite-band}
Let $(W_s)_{s\ge0}$ be a finite family of trigonometric polynomials, extended by zero, such that $\operatorname{supp}\widehat W_s\subset I_s$. Then
\begin{equation}\label{eq:finite-band}
\bigg\|\sum_{s\ge0} W_s\bigg\|_4^2 \le 16 \sum_{s\ge0}\|W_s\|_4^2.
\end{equation}
\end{lemma}
Combining this lemma with \eqref{eq:block-L4}, we obtain
\begin{equation}\label{eq:L4}
\|P_N\|_4^2 \le 16S_c^{3/2}\mathfrak n_N,\qquad
\|P_N\|_4^4 \le \frac{4000}{\pi^4 c^4}\mathfrak n_N^2.
\end{equation}

\subsection{A quantitative lower bound in \texorpdfstring{$L^1$}{L1}}\label{subsec:L1}

\begin{lemma}\label{lem:LN}
Define the constant
\begin{equation}\label{eq:Gamma}
\Gamma_c \defeq \frac{c^2\ln 2}{1024\cdot2^{5/2}5^{3/2}\pi^4[\ln(c^{-1})]^{3/2}}.
\end{equation}
For every integer $N$ satisfying $N\ge4c^{-2}$, we have
\begin{equation}\label{eq:LN-growth}
L_N \ge \Gamma_c\sqrt{\ln N}.
\end{equation}
\end{lemma}

\begin{proof}
Let $p_j/q_j$ be the continued-fraction convergents of $\alpha$, with $q_0\defeq1$. Note that $q_j\|q_j\alpha\|<1$. If $q_j\le N/2$, then the Fej\'er weight satisfies $\sigma_N(q_j)>1/2$, and hence
\[
|\widehat P_N(q_j)| = \sigma_N(q_j)|m_{q_j}| > \frac{1}{8\pi^2}.
\]
Define
\[
\mathcal{N}_\alpha(N) \defeq \#\{j\ge1 : q_j\le N/2\}.
\]
The opposite frequencies $q_j$ and $-q_j$ are distinct and contribute equally to the $L^2$ norm, so
\begin{equation}\label{eq:L2-lower}
\|P_N\|_2^2 \ge \frac{\mathcal{N}_\alpha(N)}{32\pi^4}.
\end{equation}
By H\"older's inequality $\|P_N\|_2^3 \le \|P_N\|_1\|P_N\|_4^2$, together with \eqref{eq:L4} and \eqref{eq:L2-lower}, we obtain
\begin{equation}\label{eq:LN-count}
L_N \ge \frac{\|P_N\|_2^3}{\|P_N\|_4^2} \ge \frac{c^2}{1024\cdot5^{3/2}\pi^4} \frac{\mathcal{N}_\alpha(N)^{3/2}}{1+\lfloor\log_2 N\rfloor}.
\end{equation}
Combining the constant-type condition \eqref{eq:certificate} with $\|q_j\alpha\| < q_{j+1}^{-1}$ gives, for every $j\ge1$,
\[
c \le q_j\|q_j\alpha\| < \frac{q_j}{q_{j+1}},\qquad q_{j+1} < c^{-1}q_j,\qquad q_j < c^{-j}.
\]
Set
\[
s_c(N) \defeq \max\left\{0, \left\lfloor\frac{\ln(N/2)}{\ln(c^{-1})}\right\rfloor\right\}.
\]
Then $\mathcal{N}_\alpha(N) \ge s_c(N)$. If $N\ge4c^{-2}$, then clearly $N\ge16$ and $cN/2\ge\sqrt N$, so
\[
s_c(N) \ge \frac{\ln(N/2)}{\ln(c^{-1})} - 1 = \frac{\ln(cN/2)}{\ln(c^{-1})} \ge \frac{\ln N}{2\ln(c^{-1})}.
\]
Substituting the inequality $1+\lfloor\log_2 N\rfloor \le 2\ln N/\ln 2$ into \eqref{eq:LN-count} proves \eqref{eq:LN-growth}.
\end{proof}

\begin{remark}\label{rem:partial-quotients}
If the continued-fraction partial quotients of $\alpha$ satisfy $\mathfrak a_j\le A$ for every $j\ge1$, then one may take the Markov bound $c=(A+2)^{-1}$. By the best-approximation property, for every $q_j\le q < q_{j+1}$ we have
\[
\|q\alpha\|\ge\|q_j\alpha\|>\frac{1}{q_{j+1}+q_j}\ge\frac{1}{(A+2)q}.
\]
Since $q_{j+1}\le(A+1)q_j$, it follows that
\[
\mathcal{N}_\alpha(N)\ge\max\left\{0,\left\lfloor\frac{\ln(N/2)}{\ln(A+1)}\right\rfloor\right\}.
\]
Thus, once an a priori upper bound for the partial quotients is given, all constants in the above estimates are explicit.
\end{remark}

\subsection{The endpoint multiplier and comparison with weak Dini conditions}

\begin{lemma}\label{lem:besov-multiplier}
There exists a constant $C_c<\infty$ such that
\[
\sum_{k\in I_s}|m_k|^2\le C_c\quad(s\ge0),\qquad
\|A_\alpha^{-1}w\|_\infty\le C_c\mathcal B_1(w),
\]
for every zero-mean function $w\in B^0_{2,1}(\T)$, where the inverse is normalized to have zero mean. Moreover, for all $N$ in the range of Lemma~\ref{lem:LN},
\[
\Gamma_c\sqrt{\ln N}\le L_N\le C_c\sqrt{\ln(eN)}.
\]
\end{lemma}
\begin{proof}
The argument used to prove \eqref{eq:distribution} also applies without the Fej\'er weights. Since $|m_k|\le\beta_c$,
\[
\sum_{k\in I_s}|m_k|^2
= 2\int_0^{\beta_c} t\,\#\{k\in I_s : |m_k|>t\}\,dt
\le \frac{5\beta_c}{2\pi c}.
\]
On each dyadic block, by the Cauchy--Schwarz inequality,
\[
\left\|\sum_{k\in I_s} m_k\widehat w(k)e_k\right\|_\infty
\le \left(\sum_{k\in I_s}|m_k|^2\right)^{1/2}\left(\sum_{k\in I_s}|\widehat w(k)|^2\right)^{1/2}
\le C_c\|\Pi_sw\|_2.
\]
Since $\sum_{s\ge0}\|\Pi_sw\|_2 < \infty$, the series $\sum_{s\ge0}\sum_{k\in I_s} m_k\widehat w(k)e_k$ converges uniformly to a continuous function of zero mean. Applying $A_\alpha$ in the sense of distributions recovers $w$, establishing the bound for the inverse operator.

Finally, H\"older's inequality $L_N \le \|P_N\|_2$, Parseval's identity, and the bound $0\le\sigma_N(k)\le1$ give
\[
L_N \le \|P_N\|_2 \le C_c\sqrt{1+\lfloor\log_2 N\rfloor} \le C_c\sqrt{\ln(eN)}.
\]
The lower bound is Lemma~\ref{lem:LN}.
\end{proof}

\begin{lemma}\label{lem:half-dini}
Every continuous periodic function $w$ satisfies
\[
\mathcal B_1(w) \le C\bigl(\|w\|_\infty + \mathcal D_{1/2}(w)\bigr).
\]
In particular, for every $0 \le \gamma \le 1/2$, the condition $\mathcal D_\gamma(w) < \infty$ implies $w \in B^0_{2,1}(\T)$.
\end{lemma}

\begin{proof}
Write $a_j \defeq \|\Pi_j w\|_2$ and $E_m^2 \defeq \sum_{|k|\ge2^m}|\widehat w(k)|^2$. Taking $h_m \defeq 2^{-m}$, by Plancherel's identity, we have
\[
\frac{1}{2h_m}\int_{-h_m}^{h_m}\|w(\cdot+h)-w\|_2^2\,dh
= \sum_{k\in\Z}|\widehat w(k)|^2\left(2 - \frac{\sin(2\pi kh_m)}{\pi kh_m}\right),
\]
where the term in parentheses is zero at $k=0$ and is bounded below by $2-1/\pi$ when $|k|\ge2^m$. Since the left-hand side is at most $\omega_w(2^{-m})^2$, we obtain $E_m\le(2-1/\pi)^{-1/2}\omega_w(2^{-m})$. Grouping the block indices according to $2^r\le j<2^{r+1}$ and applying the Cauchy--Schwarz inequality gives
\[
\sum_{j\ge1}a_j \le \sum_{r\ge0}2^{r/2}E_{2^r} \le C\sum_{r\ge0}2^{r/2}\omega_w(2^{-2^r}).
\]
Set $b_r \defeq (\ln 2)2^r$ and $\widetilde\omega(s) \defeq \omega_w(e^{-s})$. Since $\widetilde\omega$ is nonincreasing, for every $r\ge2$ we have
\[
\int_{b_{r-1}}^{b_r}\widetilde\omega(s)s^{-1/2}\,ds
\ge 2\bigl(\sqrt{b_r}-\sqrt{b_{r-1}}\bigr)\widetilde\omega(b_r)
= 2\sqrt{\ln 2}(1-2^{-1/2})2^{r/2}\omega_w(2^{-2^r}).
\]
The intervals $[b_{r-1},b_r]$ are pairwise disjoint and contained in $[\ln 4,\infty)\subset[1,\infty)$. Setting $t=e^{-s}$ gives
\[
\sum_{r\ge2}\int_{b_{r-1}}^{b_r}\widetilde\omega(s)s^{-1/2}\,ds
\le \int_{\ln 4}^\infty\widetilde\omega(s)s^{-1/2}\,ds
= \int_0^{1/4}\frac{\omega_w(t)}{t\sqrt{\ln(1/t)}}\,dt
\le \mathcal D_{1/2}(w).
\]
The low-frequency terms corresponding to $r\in\{0,1\}$ and the initial dyadic block term $a_0 = \|\Pi_0 w\|_2$ are all bounded by $C\|w\|_\infty$.

Since $\ln(1/t)\ge1$ for $0<t\le e^{-1}$, we have $\mathcal D_{1/2}(w) \le \mathcal D_\gamma(w)$ for every $0\le\gamma\le 1/2$. Hence $\mathcal D_\gamma(w)<\infty$ implies $\mathcal B_1(w)<\infty$ and $w\in B^0_{2,1}(\T)$.
\end{proof}

\begin{remark}
\label{rem:half-dini-strict}
The converse of Lemma~\ref{lem:half-dini} is false. Consider
\[
w_*(x)\defeq\sum_{j\ge2}j^{-3/2}\cos(2\pi2^jx),\qquad
U_*(x)\defeq-\frac1{4\pi^2}\sum_{j\ge2}j^{-3/2}2^{-2j}\cos(2\pi2^jx).
\]
Clearly, $w_*\in C^0(\T)$ and
\[
\mathcal B_1(w_*)=2^{-1/2}\sum_{j\ge2}j^{-3/2}<\infty.
\]
For every $m\ge2$, the terms of $w_*(0)-w_*(h)$ are nonnegative, and, when $j\ge m$, each cosine term has zero mean on $[0,2^{-m}]$. Thus
\[
\omega_{w_*}(2^{-m})
\ge2^m\int_0^{2^{-m}}\bigl(w_*(0)-w_*(h)\bigr)\,dh
\ge\sum_{j\ge m}j^{-3/2}\ge\frac2{\sqrt m}.
\]
On the pairwise disjoint intervals $[(m-1)\ln2,m\ln2]$, we have $\omega_{w_*}(e^{-s})\ge\omega_{w_*}(2^{-m})$ and $s^{-1/2}\ge(m\ln2)^{-1/2}$. Integrating over each interval for $m\ge3$ gives
\[
\mathcal D_{1/2}(w_*)
=\int_1^\infty\omega_{w_*}(e^{-s})s^{-1/2}\,ds
\ge2\sqrt{\ln2}\sum_{m\ge3}\frac1m=\infty.
\]

This counterexample can be realized as the second derivative of a circle diffeomorphism. Clearly, $U_*\in C^2(\T)$ and $U_*''=w_*$. Since $\|U_*'\|_\infty\le(2\pi)^{-1}\sum_{j\ge2}2^{-j}=1/(4\pi)<1$, the mapping $F_\tau(x)\defeq x+\tau+U_*(x)$ belongs to $\mathcal D^2(\T)$ for every $\tau\in\R$. Define
\[
\tau_\alpha\defeq\inf\{\tau\in\R:\rho(F_\tau)\ge\alpha\}.
\]
The monotonicity and continuity of the rotation number with respect to the parameter, together with the estimate $|\rho(F_\tau)-\tau|\le\|U_*\|_\infty$, give $\rho(F_{\tau_\alpha})=\alpha$. Consequently,
\[
F_{\tau_\alpha}''=w_*\in B^0_{2,1}(\T),\qquad
\mathcal D_{1/2}(F_{\tau_\alpha}'')=\infty.
\]
Thus Theorem~\ref{thm:rigidity} applies to this map, whereas no rigidity theorem based on the half-Dini condition is applicable.
\end{remark}

\section{Normalized potentials and their differences under rotation}\label{sec:potentials}

For $N\in\N$, define
\[
\Theta_N(y)\defeq\sgn P_N(y),\qquad
V_N(x)\defeq\int_\T P_N(x+y)\Theta_N(y)\,dy,
\]
with the convention $\sgn(0)=0$. Since $|\Theta_N|\le1$ and $P_N\Theta_N=|P_N|$, we have
\[
V_N(0)=\int_\T|P_N(y)|\,dy=L_N,
\qquad |V_N(x)|\le L_N.
\]
Since $P_N$ has zero mean, Fubini's theorem gives $\int_\T V_N=0$. The finite Fourier expansion
\[
V_N(x)=\sum_{0<|k|\le N}a_k
\left(\int_\T e_k(y)\Theta_N(y)\,dy\right)e_k(x)
\]
shows that, although $\Theta_N$ is only bounded and measurable, $V_N$ is a real trigonometric polynomial of degree at most $N$. Define
\begin{equation}\label{eq:UN}
U_N\defeq\frac{V_N}{L_N}.
\end{equation}
Then
\begin{equation}\label{eq:UN-normalization}
U_N(0)=\|U_N\|_\infty=1,
\qquad \int_\T U_N=0.
\end{equation}
The above expansion gives $\widehat V_N(k)=a_k\widehat\Theta_N(-k)$ for $0<|k|\le N$. Parseval's identity and \eqref{eq:multiplier-bounds} imply
\begin{equation}\label{eq:potential-L2}
 \|V_N\|_2^2
 \le\beta_c^2\sum_{k\in\Z}|\widehat\Theta_N(k)|^2
 =\beta_c^2\|\Theta_N\|_2^2\le\beta_c^2,
 \qquad \|U_N\|_2\le\frac{\beta_c}{L_N}.
\end{equation}

To estimate the difference under rotation, write
\[
\mathcal Q_N\defeq\Delta_\alpha P_N,
\qquad
\mathsf F_N(x)\defeq\sum_{|k|\le N}\sigma_N(k)e_k(x).
\]
The multiplier identity gives
\[
D\mathcal Q_N=\mathsf F_N-1,
\qquad
\mathsf F_N(x)=\frac1{N+1}\left|\sum_{r=0}^Ne_r(x)\right|^2.
\]
Thus $\mathsf F_N$ is the nonnegative Fej\'er kernel and satisfies $\int_\T\mathsf F_N=1$. Pairing the positive and negative Fourier frequencies of $\mathcal Q_N$ gives $\mathcal Q_N(0)=0$, and hence
\[
\mathcal Q_N(x)=\int_0^x\mathsf F_N(t)\,dt-x.
\]
For $0\le x\le1$, the integral lies in $[0,1]$, so $\|\mathcal Q_N\|_\infty\le1$. Consequently,
\[
\Delta_\alpha V_N(x)
=\int_\T\mathcal Q_N(x+y)\Theta_N(y)\,dy,
\qquad \|\Delta_\alpha V_N\|_\infty\le1.
\]
Differentiating the finite Fourier sum under the integral sign gives
\[
D\Delta_\alpha V_N(x)
=\int_\T\mathsf F_N(x+y)\Theta_N(y)\,dy
-\int_\T\Theta_N(y)\,dy.
\]
Each term on the right has absolute value at most 1. Dividing by $L_N$, we obtain
\begin{equation}\label{eq:difference-estimates}
\|\Delta_\alpha U_N\|_\infty\le\frac1{L_N},
\qquad
\|D\Delta_\alpha U_N\|_\infty\le\frac2{L_N}.
\end{equation}

\section{Analytic perturbations with two independent parameters}\label{sec:seeds}

In this section, we construct analytic conjugacies using two independent parameters: the amplitude $T$ and the Fourier degree $N$. Fix $T>0$ and define
\begin{align}
\psi_{N,T}&\defeq TU_N,
& \mathcal Z_{N,T}&\defeq\int_\T e^{\psi_{N,T}(x)}\,dx,\label{eq:potential-NT}\\
h_{N,T}(x)&\defeq \mathcal Z_{N,T}^{-1}\int_0^x e^{\psi_{N,T}(s)}\,ds,
&g_{N,T}&\defeq h_{N,T}R_\alpha h_{N,T}^{-1}.\label{eq:seed-NT}
\end{align}
Since $\psi_{N,T}$ is periodic and its exponential is strictly positive, we have $h'_{N,T} > 0$ and $h_{N,T}(x+1) = h_{N,T}(x)+1$. By the inverse function theorem, both $h_{N,T}$ and $g_{N,T}$ are real-analytic lifts of orientation-preserving homeomorphisms. The invariance of the rotation number under conjugacy gives $\rho(g_{N,T})=\alpha$; see Proposition~\ref{lem:rotation-basic}. In what follows, we fix $T$ and let $N\to\infty$; the constants may depend on $\alpha,c,T$, but are independent of $N$.

Since $\int_\T U_N=0$ and $|U_N|\le1$, the elementary inequality $1+s\le e^s\le1+s+\tfrac12e^T s^2$ for $|s|\le T$ gives
\begin{equation}\label{eq:seed-normalizer}
1\le\mathcal Z_{N,T}
\le1+\frac{\beta_c^2T^2e^T}{2L_N^2}.
\end{equation}
By Lemma~\ref{lem:LN}, we have $\mathcal Z_{N,T}\le2$ whenever $\ln N\ge\beta_c^2T^2e^T/(2\Gamma_c^2)$. Hence
\begin{equation}\label{eq:seed-slope}
h_{N,T}'(0)=\frac{e^T}{\mathcal Z_{N,T}}\ge\frac{e^T}{2}.
\end{equation}

\subsection{Estimates through the third derivative}\label{subsec:seed-distance}
Temporarily suppress the subscripts $N,T$, and write $a\defeq D\Delta_\alpha\psi$ and $\beta\defeq\Delta_\alpha\psi$. Since $h'=e^\psi/\mathcal Z$, differentiating $g\circ h=h\circ R_\alpha$ gives $(g'\circ h)h'=h'\circ R_\alpha$. Successive differentiation, together with the expression for $h'$, yields
\begin{align}
g'\circ h&=e^\beta,\label{eq:seed-first}\\
g''\circ h&=\mathcal Ze^{\beta-\psi}a,\label{eq:seed-second}\\
g'''\circ h&=\mathcal Z^2e^{\beta-2\psi}
\bigl(a'+(\beta'-\psi')a\bigr).
\label{eq:seed-third}
\end{align}
By \eqref{eq:difference-estimates},
\[
\|\beta\|_\infty\le\frac{T}{L_N},\qquad
\|a\|_\infty\le\frac{2T}{L_N},\qquad
e^{-T}\le \mathcal Z\le e^T.
\]
Since $\rho(g)=\alpha$, the periodic displacement $d(x)\defeq g(x)-x-\alpha$ must vanish at some point, so $\|d\|_\infty\le\|d'\|_\infty=\|g'-1\|_\infty$. Thus
\begin{equation}\label{eq:seed-distance}
d_2(g_{N,T},R_\alpha)
\le\max\left\{e^{T/L_N}-1,
\frac{2T}{L_N}e^{2T+T/L_N}\right\}.
\end{equation}
Write $J_N\defeq\ln N$. For $N\ge4c^{-2}$, Lemma~\ref{lem:LN} gives $L_N\ge\Gamma_c\sqrt{J_N}$. In particular, if $L_N\ge\max(1,T)$, then
\begin{equation}\label{eq:seed-C2-rate}
d_2(g_{N,T},R_\alpha)\le A_TJ_N^{-1/2},\qquad
A_T\defeq\Gamma_c^{-1}\max\{eT,2Te^{2T+1}\}.
\end{equation}
The trigonometric polynomials $a,\beta,\psi$ all have degree at most $N$. Bernstein's inequality gives
\[
\|a'\|_\infty\le\frac{4\pi NT}{L_N},\qquad
\|\beta'-\psi'\|_\infty
\le2\pi N\left(T+\frac{T}{L_N}\right).
\]
Under the same assumption $L_N\ge\max(1,T)$, we have $\mathcal Z^2e^{\beta-2\psi}\le e^{4T+1}$ and
\[
\|a'+(\beta'-\psi')a\|_\infty
\le4\pi N\left(\frac{T}{L_N}
       +\frac{T^2}{L_N}+\frac{T^2}{L_N^2}\right)
\le4\pi N(T+2T^2).
\]
Substitution into \eqref{eq:seed-third} yields
\begin{equation}\label{eq:seed-C3-rate}
\|g_{N,T}'''\|_\infty\le \mathfrak{D}_TN,
\qquad \mathfrak{D}_T\defeq4\pi e^{4T+1}(T+2T^2).
\end{equation}
For fixed $T$, the $C^2$ error tends to zero, whereas the third derivative grows at most linearly in $N$.

\subsection{Finite iterates with large derivatives}\label{subsec:iterate-selection}
In this subsection, we construct a positive integer $\ell_N$ depending only on $N$ and $\alpha$ to establish the derivative growth bound~\eqref{eq:seed-growth}. Since $U_N$ has zero mean and degree at most $N$,
\[
\sum_{r=0}^NU_N\left(\frac r{N+1}\right)=0,
\qquad U_N(0)=1.
\]
Choose rational approximations $z_{N,r}$ satisfying
\[
\left|z_{N,r}-U_N\left(\frac r{N+1}\right)\right|
\le\frac1{8N}\qquad(1\le r\le N)
\]
and let $\iota_N$ be the smallest index at which $z_{N,r}$ attains its minimum. Let $r_*$ be an index attaining $\min_{1\le r\le N}U_N(r/(N+1))$. Since this minimum is at most $-1/N$, we obtain
\[
U_N\left(\frac{\iota_N}{N+1}\right)
\le z_{N,\iota_N}+\frac1{8N}
\le z_{N,r_*}+\frac1{8N}
\le-\frac1N+\frac2{8N}=-\frac3{4N}.
\]
Thus the point $y_N\defeq\iota_N/(N+1)$ satisfies $U_N(y_N)\le -3/(4N)<0$.

By Bernstein's inequality and \eqref{eq:UN-normalization}, $\|U_N'\|_\infty\le \Lambda_N\defeq2\pi N$. Define
\begin{equation}\label{eq:MQ}
\mathfrak M_N\defeq\max\{N+1,24\Lambda_N\},\qquad
j_N\defeq2\bigl(\lfloor\log_2(\mathfrak M_N+1)\rfloor+1\bigr)+2.
\end{equation}
Let $p_j/q_j$ be the convergents of $\xi\defeq\alpha-\lfloor\alpha\rfloor$. The inequality $q_{j+2}\ge2q_j$ gives $q_{2k}\ge2^k$, so
\[
q_{j_N}\ge2^{j_N/2}
=2^{\lfloor\log_2(\mathfrak M_N+1)\rfloor+2}
>2(\mathfrak M_N+1)>\mathfrak M_N.
\]
Write $j=j_N$ and set
\[
\mathfrak b_N\defeq\left\lfloor-q_jy_N+\tfrac12\right\rfloor,
\qquad \nu_N\defeq \mathfrak b_N\pmod{q_j},\quad 0\le\nu_N<q_j.
\]
Since $\|y_N\|\ge1/(N+1)$ and $q_j>N+1$, we have $\nu_N\ne0$. As $\gcd(p_j,q_j)=1$, define $\ell_N\in\{1,\ldots,q_j-1\}$ to be the unique integer solution of the congruence
\begin{equation}\label{eq:witness-rule}
\ell_Np_j\equiv\nu_N\pmod{q_j}
\end{equation}
and note that
\[
\|\ell_N\alpha+y_N\|
<\frac1{q_{j+1}}+\frac1{2q_j}
\le\frac3{2q_j}<\frac1{16\Lambda_N}.
\]
Together with $U_N(0)=1$ and $\|U_N'\|_\infty\le\Lambda_N$, this gives $U_N(y_N+\ell_N\alpha)>15/16$. Since $U_N(y_N)<0$, we have
\begin{equation}\label{eq:seed-growth}
Dg_{N,T}^{\ell_N}(h_{N,T}(y_N))
=\exp\!\left(T[U_N(y_N+\ell_N\alpha)-U_N(y_N)]\right)
>e^{15T/16}.
\end{equation}
This establishes a lower bound for derivative growth depending only on the parameter $T$, while $g_{N,T}\to R_\alpha$ in $C^2(\T)$ as $N\to\infty$.

\section{A convergent construction with Besov control}\label{sec:construction}

Each step of the inductive construction consists of two successive perturbations. The first controls derivatives and slopes, producing quantitative growth of the derivative $\Vert{}DF_n^{\ell_n}\Vert{}_\infty$ of a finite iterate and the secant slope $\mathfrak S_n(F_n)/t_n$; the second introduces a polynomial correction that establishes a positive lower bound for the truncated sum $\mathfrak E_{M_n}(F_n)$ of the sharp dyadic frequency blocks of the second derivative. The perturbation sequence is absolutely summable in the $C^2$ topology and the target Besov topology, and these three lower bounds are preserved uniformly in open neighborhoods throughout the subsequent iterations. Finally, the real-analytic linearization theorem is used to update the conjugacy $H_n$ required for the next stage.

\subsection{Stability under iteration and conjugation}
\label{subsec:stability}

The rotation number has the following properties; see \cite{Her79}.
\begin{proposition}\label{lem:rotation-basic}
For orientation-preserving lifts $F \in \mathcal D^0(\T)$, the rotation number is monotone with respect to the pointwise partial order, continuous in the uniform topology, and invariant under conjugation by orientation-preserving homeomorphisms.
\end{proposition}

\begin{lemma}\label{lem:iteration}
Let $u,v\in\mathcal D^2(\T)$ satisfy $\|u'\|_\infty,\|v'\|_\infty\le2$ and $\|u''\|_\infty,\|v''\|_\infty\le1$. For every integer $\ell\ge1$,
\begin{equation}\label{eq:iteration-stability}
\|Du^\ell-Dv^\ell\|_\infty\le C_\ell^{\mathrm{it}}d_2(u,v),
\qquad C_\ell^{\mathrm{it}}\defeq2^{\ell-1}(2^\ell-1).
\end{equation}
Moreover,
\begin{equation}\label{eq:iterate-value-stability}
\|u^k-v^k\|_\infty\le(2^k-1)d_2(u,v),\qquad k\ge1.
\end{equation}
\end{lemma}
\begin{proof}
Write $\delta=d_2(u,v)$, and set $x_k=u^k(x)$ and $\widetilde x_k=v^k(x)$.
Induction on $k$ gives
\[
|x_k-\widetilde x_k|\le(2^k-1)\delta.
\]
Since $\|u''\|_\infty\le1$, we have
\[
|u'(x_k)-v'(\widetilde x_k)|
\le |u'(x_k)-u'(\widetilde x_k)| + |u'(\widetilde x_k)-v'(\widetilde x_k)|
\le |x_k-\widetilde x_k| + \delta
\le 2^k\delta.
\]
The chain rule then yields
\[
|Du^\ell(x)-Dv^\ell(x)|
\le\sum_{k=0}^{\ell-1}2^{\ell-1}|u'(x_k)-v'(\widetilde x_k)|
\le 2^{\ell-1}\sum_{k=0}^{\ell-1}2^k\delta
=C_\ell^{\mathrm{it}}\delta.\qedhere
\]
\end{proof}

For $H\in\mathcal D^3(\T)$, write
\[
A_j(H)\defeq\|H^{(j)}\|_\infty,
\qquad
\widetilde A_j(H)\defeq\|(H^{-1})^{(j)}\|_\infty,
\qquad 1\le j\le3,
\]
and abbreviate these quantities by $A_j$ and $\widetilde A_j$. Define
\begin{equation}\label{eq:conjugacy-constant}
C_{\mathrm{conj}}(H)\defeq\max\left\{
A_1,\ \widetilde A_1(2A_2+A_1),\
\widetilde A_1^2(4A_3+5A_2+A_1)
+\widetilde A_2(2A_2+A_1)\right\}.
\end{equation}
For any $u,v\in\mathcal D^2(\T)$ satisfying the assumptions of Lemma~\ref{lem:iteration}, we have
\begin{equation}\label{eq:conjugation-stability}
d_2(HuH^{-1},HvH^{-1})\le C_{\mathrm{conj}}(H)d_2(u,v).
\end{equation}
Write $Q=H^{-1}$ and $B_u=H\circ u$. By the chain rule,
\[
D(B_u\circ Q)=(B_u'\circ Q)Q',\qquad
D^2(B_u\circ Q)=(B_u''\circ Q)(Q')^2+(B_u'\circ Q)Q'',
\]
where $B_u'=H'(u)u'$ and $B_u''=H''(u)(u')^2+H'(u)u''$. Setting $\delta=d_2(u,v)$, we have
\begin{align*}
\|B_u-B_v\|_\infty&\le A_1\delta,\\
\|B_u'-B_v'\|_\infty&\le(2A_2+A_1)\delta,\\
\|B_u''-B_v''\|_\infty&\le(4A_3+5A_2+A_1)\delta.
\end{align*}
Substituting these bounds, together with $\|Q'\|_\infty=\widetilde A_1$ and $\|Q''\|_\infty=\widetilde A_2$, into the differentiation formulas for compositions proves~\eqref{eq:conjugation-stability}.

\begin{lemma}\label{lem:outer-Besov}
Fix $H\in\mathcal D^3(\T)$ and $T>0$, and set
\[
v_N\defeq Hg_{N,T}H^{-1}-HR_\alpha H^{-1}.
\]
There exist an explicitly bounded constant $C_{H,T}\ge1$ and a threshold $J_0(T)$ such that, whenever $N\ge4c^{-2}$ and $J_N\ge J_0(T)$,
\begin{align}
\|v_N\|_{C^2}&\le C_{H,T}J_N^{-1/2},&
\|v_N'''\|_\infty&\le C_{H,T}N,
\label{eq:outer-C2-C3}\\
\mathcal B_q(v_N'')&\le3C_{H,T}
\bigl(J_N^{1/q-1}+N^{-1}\bigr),
&&1\le q\le2,\label{eq:outer-Besov}\\
\mathcal D_\gamma(v_N'')&\le
\frac{3C_{H,T}}{1-\gamma}J_N^{1/2-\gamma},
&&\tfrac12<\gamma<1.\label{eq:outer-Dini}
\end{align}
\end{lemma}
\begin{proof}
Define
\[
J_0(T)\defeq\max\{1,\Gamma_c^{-2},(T/\Gamma_c)^2,4A_T^2\}.
\]
When $N\ge4c^{-2}$ and $J_N\ge\max\{\Gamma_c^{-2},(T/\Gamma_c)^2\}$, Lemma~\ref{lem:LN} gives $L_N\ge\Gamma_c\sqrt{J_N}\ge\max\{1,T\}$. By \eqref{eq:seed-C2-rate} and $J_N\ge4A_T^2$, we obtain
\[
d_2(g_{N,T},R_\alpha)\le A_TJ_N^{-1/2}\le\frac12.
\]
In particular, $\|g'_{N,T}\|_\infty\le 3/2$ and $\|g''_{N,T}\|_\infty\le 1/2$, so the assumptions of Lemma~\ref{lem:iteration} hold. Equation~\eqref{eq:conjugation-stability} gives
\[
\|v_N\|_{C^2}\le C_{\mathrm{conj}}(H)d_2(g_{N,T},R_\alpha)\le C_{\mathrm{conj}}(H)A_TJ_N^{-1/2}.
\]
Applying the chain rule to both terms in $v_N=H\circ g_{N,T}\circ H^{-1}-H\circ R_\alpha\circ H^{-1}$, and using $\|g'_{N,T}\|_\infty\le2$, $\|g''_{N,T}\|_\infty\le1$, $\|g'''_{N,T}\|_\infty\le \mathfrak{D}_TN$ from~\eqref{eq:seed-C3-rate}, and $N\ge1$, gives $\|v_N'''\|_\infty\le\widetilde C_{H,T}N$, where
\[
\widetilde C_{H,T}\defeq \widetilde A_1^3(9A_3+6A_2+A_1\mathfrak{D}_T)+3\widetilde A_1\widetilde A_2(5A_2+A_1)+3A_1\widetilde A_3.
\]
Taking
\[
C_{H,T}\defeq\max\{1,C_{\mathrm{conj}}(H)A_T,\widetilde C_{H,T}\},
\]
we obtain $\|v_N'''\|_\infty\le C_{H,T}N$.

Write $w\defeq v_N''$. By \eqref{eq:outer-C2-C3}, $\|w\|_2\le\|v_N\|_{C^2}\le C_{H,T}J_N^{-1/2}$ and $\|w'\|_2\le\|v_N'''\|_\infty\le C_{H,T}N$. Set $M\defeq\lceil 2\log_2N\rceil$. Then $2^{-M}\le N^{-2}$; since $J_N\ge1$, we also have $M+1\le(2/\ln 2+2)J_N<5J_N$. For $1\le q\le2$, the discrete H\"older inequality and the orthogonality of the dyadic frequency blocks give
\[
\left(\sum_{s=0}^M\|\Pi_sw\|_2^q\right)^{1/q}
\le(M+1)^{1/q-1/2}\left(\sum_{s=0}^M\|\Pi_sw\|_2^2\right)^{1/2}
\le\sqrt5C_{H,T}J_N^{1/q-1}.
\]
For $s>M$, the Bernstein estimate on the frequency band $2^s\le|k|<2^{s+1}$ gives $\|\Pi_s w\|_2 \le (2\pi 2^s)^{-1}\|w'\|_2$. Since $q\ge 1$,
\[
\left(\sum_{s=M+1}^\infty\|\Pi_sw\|_2^q\right)^{1/q}
\le \sum_{s=M+1}^\infty\|\Pi_sw\|_2
\le\frac{C_{H,T}}{2\pi}\sum_{s=M+1}^\infty 2^{-s}
\le\frac{C_{H,T}}{2\pi}N^{-1}.
\]
Combining this with the low-frequency estimate proves~\eqref{eq:outer-Besov}.

Finally, we have $\omega_w(t)\le\min\{2\|w\|_\infty,\|w'\|_\infty t\}\le C_{H,T}\min\{2J_N^{-1/2},Nt\}$. Setting $s=\ln(1/t)$ and splitting the integral at $s=\ln N=J_N$, we obtain
\[
\mathcal D_\gamma(w)
=\int_1^\infty\omega_w(e^{-s})s^{-\gamma}\,ds
\le C_{H,T}\left(\frac2{\sqrt{J_N}}\int_1^{J_N}s^{-\gamma}\,ds+\int_{J_N}^\infty e^{J_N-s}s^{-\gamma}\,ds\right).
\]
For $1/2<\gamma<1$, the first integral term is strictly less than $\frac{2}{1-\gamma}J_N^{1/2-\gamma}$. For the second, the inequality $s^{-\gamma}\le J_N^{-\gamma}$ gives
\[
\int_{J_N}^\infty e^{J_N-s}s^{-\gamma}\,ds
\le J_N^{-\gamma}\int_{J_N}^\infty e^{J_N-s}\,ds
= J_N^{-\gamma}
\le \frac{1}{1-\gamma}J_N^{1/2-\gamma}.
\]
Combining the two terms gives~\eqref{eq:outer-Dini}. For $\gamma=1$, the same argument yields
\[
\mathcal D_1(w)\le\frac{3C_{H,T}(1+\ln J_N)}{\sqrt{J_N}}.\qedhere
\]
\end{proof}

\subsection{Regularity of the normalized conjugacy}
\label{subsec:conjugacy-regularity}

We use the convention that the forward conjugacy satisfies
\begin{equation}\label{eq:forward-conjugacy}
 F\circ h=h\circ R_\alpha,\qquad h(0)=0.
\end{equation}
The displacements $h-\id$ and $h^{-1}-\id$ are periodic functions. The conjugacy denoted by $h$ in \cite{KO89} corresponds to the inverse map $h^{-1}$ used here.

For the continued fraction denominators $q_n$ introduced in Lemma~\ref{lem:LN}, we follow the notation of Katznelson and Ornstein and define
\begin{equation}\label{eq:KO-return-distortion}
 K_n^0\defeq\|\ln DF^{q_n}\|_\infty, \qquad n\ge1.
\end{equation}
The results from \cite{KO89D,KO89} needed below are collected in Appendix~\ref{app:KO}.

\begin{lemma}\label{lem:KO-dyadic}
Let $F\in\mathcal F_\alpha^2$, where $\alpha$ is an irrational number of constant type, and let $h$ satisfy~\eqref{eq:forward-conjugacy}. Then both $h$ and $h^{-1}$ are absolutely continuous. Moreover, there exists a nonnegative sequence $(\beta_n)_{n\ge1}\in\ell^2$ such that, for every $x\in\R$,
\begin{equation}\label{eq:KO-dyadic-ratio}
 \left|\frac{h(x+2^{-n})-h(x)}{h(x)-h(x-2^{-n})}-1\right|
 \le \beta_n.
\end{equation}
\end{lemma}
\begin{proof}
The absolute continuity of $h$ and $h^{-1}$ and the derivative regularity $h',(h^{-1})'\in L^2(\T)$ follow from Proposition~\ref{appC:KO-interval-ratios}.
Proposition~\ref{appC:KO-square-summability} gives $\sum_{m\ge1}(K_m^0)^2<\infty$. Let $n_*,C,\theta$ and the index sequence $\ell(n)$ be as in Proposition~\ref{appC:KO-interval-ratios}. For $n\ge n_*$, define
\[
 \beta_n\defeq C\sum_{m\ge\ell(n)}\theta^{m-\ell(n)}K_m^0.
\]
Then~\eqref{eq:KO-dyadic-ratio} follows directly from Proposition~\ref{appC:KO-interval-ratios}. For $1\le n<n_*$, let $\beta_n$ be the supremum over $x\in[0,1]$ of the left-hand side of~\eqref{eq:KO-dyadic-ratio}. Each $\beta_n$ is finite by the strict monotonicity and continuity of $h$.

For $n\ge n_*$, by the Cauchy--Schwarz inequality, we have
\[
 \beta_n^2\le C^2\left(\sum_{m\ge\ell(n)}\theta^{m-\ell(n)}\right)\left(\sum_{m\ge\ell(n)}\theta^{m-\ell(n)}(K_m^0)^2\right)
 =\frac{C^2}{1-\theta}\sum_{m\ge\ell(n)}\theta^{m-\ell(n)}(K_m^0)^2.
\]
By Proposition~\ref{appC:KO-interval-ratios}, for every $k\in\N$ there are at most $L_A$ indices $n\ge n_*$ such that $\ell(n)=k$. Hence
\[
 \sum_{n\ge n_*}\beta_n^2
 \le\frac{C^2}{1-\theta}\sum_{m\ge1}(K_m^0)^2
       \sum_{\substack{n\ge n_*\\\ell(n)\le m}}\theta^{m-\ell(n)}
 \le\frac{C^2L_A}{(1-\theta)^2}\sum_{m\ge1}(K_m^0)^2<\infty.
\]
Together with $\sum_{1\le n<n_*}\beta_n^2 < \infty$, this yields $(\beta_n)_{n\ge1}\in\ell^2$.\qedhere
\end{proof}

\begin{proposition}\label{prop:conjugacy-regularity}
Under the assumptions of Lemma~\ref{lem:KO-dyadic}, for every real number $r$ and every $1<p<\infty$, we have
\begin{align}
 &h'>0\quad\text{almost everywhere},\qquad
   \int_\T(h')^r\,dx<\infty,\label{eq:all-derivative-moments}\\
 &h'\in B^0_{p,2}(\T),\qquad
   h-\id,\ h^{-1}-\id\in W^{1,p}(\T).
   \label{eq:conjugacy-positive-regularity}
\end{align}
Here $h'\in B^0_{p,2}(\T)$ is equivalent to $h-\id\in B^1_{p,2}(\T)$.
\end{proposition}
\begin{proof}
For $n\ge0$ and $k\in\Z$, write $I_{n,k}\defeq [k2^{-n},(k+1)2^{-n})$. On each $I_{n,k}$, define
\[
w_n(x)\defeq 2^n\big(h((k+1)2^{-n})-h(k2^{-n})\big),\qquad x\in I_{n,k}.
\]
Since $h(x+1)=h(x)+1$, the function $w_n$ is 1-periodic and $w_0\equiv 1$. Since $h$ is strictly increasing, $w_n$ is positive. For each parent interval $I=I_{n-1,k}$, denote its left and right children of equal length by $I^-$ and $I^+$, respectively. By construction,
\[
\frac{w_n|_{I^-}+w_n|_{I^+}}{2}=w_{n-1}|_I.
\]
Thus we may write $w_n|_{I^+}=w_{n-1}|_I(1+a)$ and $w_n|_{I^-}=w_{n-1}|_I(1-a)$, where the parameter is real and $|a|<1$. Applying~\eqref{eq:KO-dyadic-ratio} at the common endpoint $(2k+1)2^{-n}$ of the two children gives $|a|\le\frac{\beta_n}{2-\beta_n}$ whenever $\beta_n<1$.

Since $(\beta_n)_{n\ge1}\in\ell^2$, there exists $n_0\ge1$ such that $\beta_n<1$ and $\frac{\beta_n}{2-\beta_n}\le 1/2$ for every $n\ge n_0$. Define
\[
\zeta_n\defeq
\begin{cases}
\max_I |a|, & 1\le n<n_0,\\
\dfrac{\beta_n}{2-\beta_n}, & n\ge n_0,
\end{cases}.
\]
For $n<n_0$, the maximum is taken over all parent intervals at level $n-1$. Since $h$ is strictly increasing, $\zeta_n<1$. It follows that $\sum_{n\ge1}\zeta_n^2<\infty$ and that $|a|\le\zeta_n$ for every level and every parent interval.

For any fixed $r \in \R$, Taylor expansion gives, on $[-1/2,1/2]$,
\[
 \frac{(1+a)^r+(1-a)^r}{2}\le 1+C_r a^2.
\]
Integrating over each parent interval and summing gives $\int_\T w_n^r\,dx \le (1+C_r\zeta_n^2)\int_\T w_{n-1}^r\,dx$. Iteration yields
\begin{equation}\label{eq:dyadic-moment-bound}
 \sup_{n\ge0}\int_\T w_n^r\,dx
 \le C_{F,r}\prod_{n\ge n_0}(1+C_r\zeta_n^2)<\infty.
\end{equation}
Since $h$ is absolutely continuous, $w_n$ is the dyadic conditional expectation of $h'$. By the Lebesgue differentiation theorem, $w_n\to h'$ almost everywhere. Fatou's lemma gives, for every $r$,
\[
 \int_\T(h')^r\,dx \le \liminf_{n\to\infty}\int_\T w_n^r\,dx < \infty.
\]
Since $h$ preserves orientation, $h'\ge0$ almost everywhere. Combining this with $\int_\T (h')^{-1}\,dx < \infty$ for $r=-1$ gives $h'>0$ almost everywhere, proving~\eqref{eq:all-derivative-moments}. Setting $y=h(x)$, we obtain, for every $1<p<\infty$,
\begin{equation}\label{eq:inverse-moments}
 \int_\T((h^{-1})'(y))^p\,dy
 =\int_\T(h'(x))^{1-p}\,dx<\infty.
\end{equation}
This proves that $h-\id, h^{-1}-\id\in W^{1,p}(\T)$.

For $n\ge1$, let $d_n\defeq w_n-w_{n-1}$. By $|d_n|\le \zeta_n w_{n-1}$ and~\eqref{eq:dyadic-moment-bound},
\begin{equation}\label{eq:dyadic-increment-Lp}
 \|d_n\|_p\le \zeta_n\|w_{n-1}\|_p\le C_{F,p}\zeta_n.
\end{equation}
Since $d_n$ is constant on each dyadic interval of length $2^{-n}$ and has mean zero on each parent interval, the fact that $h'\in L^p(\T)$ and Proposition~\ref{appB:dyadic-step} imply
\[
 \|h'\|_{B^0_{p,2}}
 \le C_p\left(\|h'\|_p+
        \Bigl(\sum_{n\ge1}\|d_n\|_p^2\Bigr)^{1/2}\right)<\infty.
\]
Equation~\eqref{appB:lift-convention} now gives $h-\id\in B^1_{p,2}(\T)$.\qedhere
\end{proof}

\subsection{The two perturbations at each stage}
\label{subsec:stage-perturbations}

Lemma~\ref{lem:outer-Besov} provides the estimates for the first perturbation. The second perturbation is introduced to increase the local norms of the sharp Fourier blocks while retaining control in weaker topologies. Define the real-valued trigonometric polynomial of mean zero
\[
 \mathfrak p_N\defeq A_\alpha U_N,\qquad \deg\mathfrak p_N\le N.
\]
By~\eqref{eq:difference-estimates} and Lemma~\ref{lem:LN}, whenever $N\ge\max\{3,4c^{-2}\}$, we have
\begin{equation}\label{eq:insertion-basic}
 \|\mathfrak p_N\|_\infty\le\frac2{L_N}\le C_cJ_N^{-1/2}.
\end{equation}
Since $\mathfrak p_N$ has at most $1+\lfloor\log_2N\rfloor\le C J_N$ nonzero dyadic frequency blocks, for $1<q\le2$ we have
\begin{equation}\label{eq:insertion-besov}
 \mathcal B_q(\mathfrak p_N)
 \le(CJ_N)^{1/q-1/2}\|\mathfrak p_N\|_2
 \le C_{c,q}J_N^{1/q-1}\longrightarrow0.
\end{equation}
For $2<q\le\infty$, we have
\begin{equation}\label{eq:insertion-besov-high}
 \mathcal B_q(\mathfrak p_N)\le\|\mathfrak p_N\|_2\le\|\mathfrak p_N\|_\infty
 \le C_cJ_N^{-1/2}\longrightarrow0.
\end{equation}
The Bernstein estimate $\|\mathfrak p_N'\|_\infty\le 2\pi N\|\mathfrak p_N\|_\infty$ gives
\[
 \omega_{\mathfrak p_N}(t)\le\min\{2\|\mathfrak p_N\|_\infty, t\|\mathfrak p_N'\|_\infty\}
 \le C_cJ_N^{-1/2}\min\{1,Nt\}.
\]
Integrating separately over $[0,N^{-1}]$ and $[N^{-1},e^{-1}]$, we obtain
\begin{equation}\label{eq:insertion-dini}
 \mathcal D_\gamma(\mathfrak p_N)\le C_{c,\gamma}
 \begin{cases}
 J_N^{1/2-\gamma},&1/2<\gamma<1,\\
 J_N^{-1/2}\ln(eJ_N),&\gamma=1,
 \end{cases}
 \longrightarrow0.
\end{equation}

Since $\mathfrak p_N$ has mean zero and $A_\alpha^{-1}\mathfrak p_N=U_N$, Lemma~\ref{lem:besov-multiplier} gives
\begin{equation}\label{eq:insertion-endpoint-lower}
\mathcal B_1(\mathfrak p_N)\ge C_c^{-1}\qquad(N\ge3).
\end{equation}
For a $C^2$ lift $G$, define
\begin{equation}\label{eq:finite-endpoint-functional}
\mathfrak E_M(G)\defeq\sum_{s=0}^M\|\Pi_sG''\|_2,
\end{equation}
which satisfies the stability estimate
\begin{equation}\label{eq:endpoint-stability}
|\mathfrak E_M(G)-\mathfrak E_M(\widetilde G)|
\le(M+1)d_2(G,\widetilde G).
\end{equation}

Let $G$ be a real-analytic orientation-preserving lift with $\rho(G)=\alpha$. Given any $L,\epsilon_C,\epsilon_B,\epsilon_D>0$, choose $\lambda>0$ such that
\begin{equation}\label{eq:lambda-choice}
\lambda C_c^{-1}>\mathcal B_1(G'')+4L.
\end{equation}
For any function $f\in C^0(\T)$ of mean zero, define
\[
(\mathcal I f)(x)\defeq \int_0^x f(t)\,dt - \int_0^1\left(\int_0^s f(t)\,dt\right)ds.
\]
Clearly, $(\mathcal I f)'=f$, $\int_\T \mathcal I f\,dx=0$, and $\|\mathcal I f\|_\infty\le 2\|f\|_\infty$. Set
\[
z_N\defeq \lambda \mathcal I^2(\mathfrak p_N).
\]
Then $z_N$ is again a real trigonometric polynomial of mean zero and degree at most $N$, and it satisfies
\begin{equation}\label{eq:primitive-smallness}
z_N''=\lambda \mathfrak p_N,\qquad
\|z_N'\|_\infty\le2\lambda\|\mathfrak p_N\|_\infty,\qquad
\|z_N\|_{C^2}\le 4\lambda\|\mathfrak p_N\|_\infty\longrightarrow0.
\end{equation}
With $\lambda$ fixed, equations~\eqref{eq:insertion-besov}, \eqref{eq:insertion-dini}, and~\eqref{eq:primitive-smallness} allow us to choose a sufficiently large truncation degree $N$ such that, for any fixed $q>1$ and $\gamma\in(1/2,1]$,
\begin{equation}\label{eq:insertion-budgets}
8\lambda\|\mathfrak p_N\|_\infty<\epsilon_C,\qquad
\lambda\mathcal B_q(\mathfrak p_N)<\epsilon_B,\qquad
\lambda\mathcal D_\gamma(\mathfrak p_N)<\epsilon_D.
\end{equation}

Further require $N$ to be large enough that $2\lambda\|\mathfrak p_N\|_\infty<\tfrac12\min G'$. Then, for every constant $\eta \in \R$, the lift $G+z_N+\eta$ satisfies $(G+z_N+\eta)'\ge\tfrac12\min G'>0$. The pointwise bounds
\[
G-\|z_N\|_\infty\le G+z_N\le G+\|z_N\|_\infty
\]
and the monotonicity and continuity in Proposition~\ref{lem:rotation-basic} imply that there exists a constant $\tau_N$ with $|\tau_N|\le\|z_N\|_\infty$ such that
\begin{equation}\label{eq:rotation-correction}
\widetilde G_N\defeq G+z_N+\tau_N
\end{equation}
has the same rotation number $\rho(\widetilde G_N)=\alpha$. By~\eqref{eq:primitive-smallness} and~\eqref{eq:insertion-budgets},
\[
d_2(\widetilde G_N,G)
\le\|z_N\|_{C^2}+|\tau_N|
\le2\|z_N\|_{C^2}
\le8\lambda\|\mathfrak p_N\|_\infty<\epsilon_C.
\]
Let $M(N)\defeq\lfloor\log_2N\rfloor$. All frequencies of $\mathfrak p_N$ lie in blocks with $s\le M(N)$, so $\sum_{s=0}^{M(N)}\|\Pi_s\mathfrak p_N\|_2=\mathcal B_1(\mathfrak p_N)$. Since $\widetilde G_N''=G''+\lambda\mathfrak p_N$, applying the reverse triangle inequality on each of these finitely many blocks gives
\begin{align*}
\mathfrak E_{M(N)}(\widetilde G_N)
&=\sum_{s=0}^{M(N)}\|\Pi_s(G''+\lambda\mathfrak p_N)\|_2\\
&\ge\lambda\mathcal B_1(\mathfrak p_N)-\sum_{s=0}^{M(N)}\|\Pi_sG''\|_2
\ge\lambda\mathcal B_1(\mathfrak p_N)-\mathcal B_1(G'').
\end{align*}
Combining~\eqref{eq:insertion-endpoint-lower} and~\eqref{eq:lambda-choice}, we obtain
\begin{equation}\label{eq:finite-endpoint-witness}
\mathfrak E_{M(N)}(\widetilde G_N)
\ge \lambda C_c^{-1}-\mathcal B_1(G'')
>4L.
\end{equation}
Moreover, by choosing a smaller $\epsilon_C>0$ in~\eqref{eq:insertion-budgets}, the construction ensures that $\widetilde G_N$ remains in any prescribed $C^2$ open neighborhood of $G$.

\subsection{The recursive construction}
\label{subsec:recursion}

\begin{proof}[Proof of Theorem~\ref{thm:main}]
Fix $q_*>1$, $\gamma_*\in(1/2,1]$, and $\varepsilon>0$. Take $\theta_*\defeq\min\{1/4,\gamma_*-1/2\}$, and set
\begin{equation}\label{eq:exponent-budgets}
\tau\defeq\min\{1,\varepsilon\},\qquad
\overline q\defeq\min\{q_*,2\},\qquad
b_n\defeq\varepsilon2^{-n-4}\qquad(n\ge1),
\end{equation}
and
\begin{equation}\label{eq:gamma-sequence}
\mathfrak q_n\defeq1+\frac{\overline q-1}{n+1},\qquad
\gamma_n\defeq\frac12+\frac{\theta_*}{n}\in(1/2,\gamma_*]\qquad(n\ge1).
\end{equation}

\medskip\noindent\textbf{Inductive construction.}
Start with $F_0\defeq R_\alpha$, $H_0\defeq\id$, and $r_0\defeq\tau$. For $n\ge1$, suppose inductively that positive numbers $r_0,\ldots,r_{n-1}>0$ and a real-analytic orientation-preserving lift $H_{n-1}$ with $H_{n-1}(0)=0$ have been chosen such that
\begin{equation}\label{eq:analytic-representation}
F_{n-1}=H_{n-1}R_\alpha H_{n-1}^{-1}.
\end{equation}
Set
\begin{equation}\label{eq:stage-budgets}
\delta_n\defeq2^{-n-8}\min_{0\le j<n} r_j.
\end{equation}

Write
\begin{equation}\label{eq:outer-constants}
\kappa_{n-1}\defeq\frac{\max H_{n-1}'}{\min H_{n-1}'},\qquad
T_n\defeq\frac{16}{15}\ln(4\kappa_{n-1}n)+1,
\end{equation}
and set $C_n\defeq C_{H_{n-1},T_n}$. Introduce the threshold
\begin{equation}\label{eq:degree-choice}
\begin{split}
S_n\defeq1+\max\biggl\{&J_0(T_n),\ \ln(4c^{-2}),\
(2C_n/\delta_n)^2,\
\frac{\beta_c^2T_n^2e^{T_n}}{2\Gamma_c^2},\\
&(12C_n/b_n)^{\mathfrak q_n/(\mathfrak q_n-1)},\ \ln(12C_n/b_n),\
\left(\frac{6C_n}{(1-\gamma_n)b_n}\right)^{1/(\gamma_n-1/2)}
\biggr\},
\end{split}
\end{equation}
Take $N_n\defeq 2^{\lfloor S_n/\ln2\rfloor+1}$. Since $S_n<\ln N_n\le S_n+\ln2$, we have $\mathcal Z_{N_n,T_n}\le2$ and
\begin{equation}\label{eq:degree-rule}
\begin{split}
C_nJ_{N_n}^{-1/2}&<\frac{\delta_n}{2},\\
3C_n\bigl(J_{N_n}^{1/\mathfrak q_n-1}+N_n^{-1}\bigr)&<\frac{b_n}{2},\\
\frac{3C_n}{1-\gamma_n}J_{N_n}^{1/2-\gamma_n}&<\frac{b_n}{2}.
\end{split}
\end{equation}
Define
\begin{equation}\label{eq:large-step}
G_n\defeq H_{n-1}g_{N_n,T_n}H_{n-1}^{-1},\qquad \ell_n\defeq\ell_{N_n}.
\end{equation}
By Lemma~\ref{lem:outer-Besov} and~\eqref{eq:degree-rule},
\begin{equation}\label{eq:large-step-budgets}
d_2(G_n,F_{n-1})<\frac{\delta_n}{2},\quad
\mathcal B_{\mathfrak q_n}(G_n''-F_{n-1}'')<\frac{b_n}{2},\quad
\mathcal D_{\gamma_n}(G_n''-F_{n-1}'')<\frac{b_n}{2}.
\end{equation}
Moreover, by~\eqref{eq:seed-growth} and~\eqref{eq:outer-constants},
\begin{equation}\label{eq:new-iterate-witness}
\|DG_n^{\ell_n}\|_\infty
\ge\kappa_{n-1}^{-1}e^{15T_n/16}>4n.
\end{equation}

Set $K_n\defeq H_{n-1}\circ h_{N_n,T_n}$. Then $G_n=K_nR_\alpha K_n^{-1}$ and $K_n(0)=0$. Combining $\min H_{n-1}'\ge\kappa_{n-1}^{-1}$, $h_{N_n,T_n}'(0)\ge\tfrac12e^{T_n}$, and~\eqref{eq:outer-constants}, we obtain
\[
K_n'(0)\ge\frac{e^{T_n}}{2\kappa_{n-1}}>4n.
\]
Thus $K_n(t)>4nt$ in a sufficiently small right neighborhood of the origin. Since the positive orbit $\{k\alpha\}_{k\ge1}$ accumulates at the origin, there exists an integer $k_n\ge1$ such that
\begin{equation}\label{eq:new-secant-witness}
t_n\defeq k_n\alpha-\lfloor k_n\alpha\rfloor\in(0,2^{-n})
\end{equation}
and $K_n(t_n)>4nt_n$. Introduce the functional
\begin{equation}\label{eq:secant-functional}
\mathfrak S_n(u)\defeq u^{k_n}(0)-\lfloor k_n\alpha\rfloor,
\end{equation}
so that $\mathfrak S_n(G_n)=K_n(t_n)>4nt_n$. Moreover, Lemma~\ref{lem:iteration} gives
\begin{equation}\label{eq:secant-stability}
|\mathfrak S_n(u)-\mathfrak S_n(v)|\le(2^{k_n}-1)d_2(u,v).
\end{equation}

Apply the perturbation construction in Subsection~\ref{subsec:stage-perturbations} with $G=G_n$ and $L=n$. By~\eqref{eq:secant-stability} and Lemma~\ref{lem:iteration}, to ensure that
\begin{equation}\label{eq:witness-preservation}
\|DF_n^{\ell_n}\|_\infty>3n,\qquad \mathfrak S_n(F_n)>3nt_n
\end{equation}
it suffices to require the $C^2$ perturbation to satisfy
\begin{equation}\label{eq:perturb-tolerance}
d_2(F_n,G_n)<\min\left\{\frac{\delta_n}{2},\,\frac{\|DG_n^{\ell_n}\|_\infty-3n}{C_{\ell_n}^{\mathrm{it}}},\,\frac{\mathfrak S_n(G_n)-3nt_n}{2^{k_n}-1}\right\}.
\end{equation}
Accordingly, choose $\lambda_n$ as in~\eqref{eq:lambda-choice} and then choose $\widehat N_n$ sufficiently large. The resulting $F_n$ satisfies
\begin{equation}\label{eq:insertion-stage-budgets}
d_2(F_n,G_n)<\frac{\delta_n}{2},\quad
\mathcal B_{\mathfrak q_n}(F_n''-G_n'')<\frac{b_n}{2},\quad
\mathcal D_{\gamma_n}(F_n''-G_n'')<\frac{b_n}{2},
\end{equation}
as well as~\eqref{eq:witness-preservation}.

Set $M_n\defeq\lfloor\log_2\widehat N_n\rfloor$. By~\eqref{eq:finite-endpoint-witness} and~\eqref{eq:insertion-stage-budgets},
\begin{equation}\label{eq:stage-witnesses}
\mathfrak E_{M_n}(F_n)>3n,\qquad
\|DF_n^{\ell_n}\|_\infty>3n,\qquad
\mathfrak S_n(F_n)>3nt_n.
\end{equation}
Since $d_2(F_n,R_\alpha)\le\sum_{j=1}^n\delta_j\le1/256$, we use~\eqref{eq:iteration-stability}, \eqref{eq:endpoint-stability}, and~\eqref{eq:secant-stability} to choose a perturbation radius
\begin{equation}\label{eq:witness-radius}
0<r_n<\min\left\{\frac14,\,
\frac{\|DF_n^{\ell_n}\|_\infty-2n}{C_{\ell_n}^{\mathrm{it}}},\,
\frac{\mathfrak E_{M_n}(F_n)-2n}{M_n+1},\,
\frac{\mathfrak S_n(F_n)-2nt_n}{2^{k_n}-1}\right\}.
\end{equation}
For every orientation-preserving lift $u$ satisfying $d_2(u,F_n)<r_n$, we have $d_2(u,R_\alpha)<1/256+1/4<1/2$. The derivative bounds in Lemma~\ref{lem:iteration} apply to $u$, and
\begin{equation}\label{eq:open-ball-preservation}
\mathfrak E_{M_n}(u)>2n,\qquad
\|Du^{\ell_n}\|_\infty>2n,\qquad
\mathfrak S_n(u)>2nt_n.
\end{equation}

Since $F_n$ is real-analytic and orientation-preserving with $\rho(F_n)=\alpha$, the Herman--Yoccoz global analytic linearization theorem \cite[Corollary, p.~335]{Yoc84} gives a real-analytic orientation-preserving lift $H_n$ satisfying $H_n(0)=0$ and
\begin{equation}\label{eq:analytic-representation-next}
F_n=H_nR_\alpha H_n^{-1}.
\end{equation}
This completes the $n$th step of the inductive construction.

\medskip\noindent\textbf{The limit map.}
Combining~\eqref{eq:large-step-budgets} and~\eqref{eq:insertion-stage-budgets}, we obtain, for every $n\ge1$,
\begin{equation}\label{eq:combined-stage-budgets}
d_2(F_n,F_{n-1})<\delta_n,\quad
\mathcal B_{\mathfrak q_n}(F_n''-F_{n-1}'')<b_n,\quad
\mathcal D_{\gamma_n}(F_n''-F_{n-1}'')<b_n.
\end{equation}
Since $\sum_{n\ge1}\delta_n\le1/256<\infty$, the sequence $\{F_n\}$ converges in the $C^2$ topology to an orientation-preserving lift $F$ with $\min F'>0$. By continuity of the rotation number in the $C^0$ topology, $\rho(F)=\lim_{n\to\infty}\rho(F_n)=\alpha$. Moreover, by~\eqref{eq:stage-budgets}, we have $\delta_j\le2^{-j-8}r_n$ for every $j\ge n+1$, and hence
\begin{equation}\label{eq:tail}
d_2(F,F_n)\le\sum_{j=n+1}^\infty \delta_j\le r_n\sum_{j=n+1}^\infty 2^{-j-8}=2^{-n-8}r_n<r_n.
\end{equation}
Together with~\eqref{eq:open-ball-preservation}, this gives
\begin{equation}\label{eq:limit-witnesses}
\|DF^{\ell_n}\|_\infty>2n,\qquad
\mathcal B_1(F'')\ge \mathfrak E_{M_n}(F)>2n,\qquad
\mathfrak S_n(F)>2nt_n.
\end{equation}
Letting $n\to\infty$, we conclude that $\{\|DF^{\ell_n}\|_\infty\}$ is unbounded and $F''\notin B^0_{2,1}(\T)$.

Since $\mathfrak q_n\le q_*$ and $\gamma_n\le\gamma_*$ for all $n\ge1$, \eqref{eq:combined-stage-budgets} gives
\begin{equation}\label{eq:limit-smallness}
d_2(F,R_\alpha)+\max\{\mathcal B_{q_*}(F''),\,\mathcal D_{\gamma_*}(F'')\}
\le\sum_{n=1}^\infty(\delta_n+b_n)
\le\frac{\tau}{256}+\frac{\varepsilon}{16}<\varepsilon.
\end{equation}
Let $q>1$ be arbitrary. Since $\mathfrak q_n\to1$, choose $n_q\ge1$ such that $\mathfrak q_n\le q$ for every $n\ge n_q$. As $\sum_{n\ge1}b_n<\infty$, the series $\sum_{n\ge n_q}(F_n''-F_{n-1}'')$ converges absolutely in $B^0_{2,q}(\T)$. Together with the real analyticity of $F_{n_q-1}''$, this gives
\begin{equation}\label{eq:limit-all-Besov}
F''\in\bigcap_{q>1}B^0_{2,q}(\T).
\end{equation}
Similarly, let $\gamma\in(1/2,1]$ be arbitrary. For sufficiently large $n$, we have $\gamma_n\le\gamma$. The inequality $\mathcal D_\gamma\le\mathcal D_{\gamma_n}$ and Fatou's lemma yield
\begin{equation}\label{eq:limit-all-Dini}
\mathcal D_\gamma(F'')<\infty\qquad(1/2<\gamma\le1).
\end{equation}

\medskip\noindent\textbf{The conjugacy.}
By Denjoy's theorem, there exists a unique orientation-preserving homeomorphism $h$ satisfying $F\circ h=h\circ R_\alpha$ and $h(0)=0$. Since $F^{k_n}(0)=h(k_n\alpha)=h(t_n)+\lfloor k_n\alpha\rfloor$, equations~\eqref{eq:secant-functional} and~\eqref{eq:limit-witnesses} give
\begin{equation}\label{eq:direct-non-Lipschitz}
\frac{h(t_n)-h(0)}{t_n}=\frac{\mathfrak S_n(F)}{t_n}>2n\qquad(0<t_n<2^{-n}).
\end{equation}
Since $t_n\to 0$, the right difference quotients at the origin diverge, so $h\notin\mathrm{Lip}(\T)$. More generally, $F$ and $R_\alpha$ are not $C^1$ conjugate: if there were $\widetilde h\in\mathcal D^1(\T)$ such that $F=\widetilde hR_\alpha\widetilde h^{-1}$, then for all $k\ge1$ we would have
\begin{equation}\label{eq:c1-obstruction}
\|DF^k\|_\infty\le\frac{\max\widetilde h'}{\min\widetilde h'}<\infty,
\end{equation}
contradicting the unboundedness of $\{\|DF^{\ell_n}\|_\infty\}$ in~\eqref{eq:limit-witnesses}.

On the other hand, applying Proposition~\ref{prop:conjugacy-regularity} to $F$ shows that $h$ satisfies $h'>0$ almost everywhere and
\begin{equation}\label{eq:conjugacy-regularity-spaces}
h'\in\bigcap_{1<p<\infty}B^0_{p,2}(\T),\qquad
h-\id,\ h^{-1}-\id\in\bigcap_{1<p<\infty}W^{1,p}(\T).
\end{equation}
In particular, $h$ is absolutely continuous, and~\eqref{eq:direct-non-Lipschitz} gives
\begin{equation}\label{eq:ess-sup-blowup}
\frac{1}{t_n}\int_0^{t_n}h'(x)\,dx=\frac{h(t_n)}{t_n}>2n.
\end{equation}
Thus $\operatorname*{ess\,sup}h'=\infty$, that is, $h'\notin L^\infty(\T)$.
\end{proof}

\begin{remark}\label{rem:explicit}
The construction proceeds by a step-by-step recursion; the parameter choices at a single stage are shown in Figure~\ref{fig:recursion}. The perturbation amplitudes are fixed before the truncation degree $\widehat N_n$ is chosen. For each fixed $q>1$ and $1/2<\gamma\le1$, the corresponding Besov and weak Dini tail estimates hold from some stage onward.

Once the conjugacy $H_{n-1}$ and its derivative distortion bound are given, all kernel estimates and the rules~\eqref{eq:witness-rule} and~\eqref{eq:degree-choice} are fully quantitative. The representation $F_n=H_nR_\alpha H_n^{-1}$ relies on the Herman--Yoccoz global analytic linearization theorem \cite[Corollary, p.~335]{Yoc84}. To turn this procedure into a fully quantitative effective construction, the key requirement is an effective a priori bound on the derivative of the conjugacy $H_n$. Given this control, the parameters $S_{n+1}, N_{n+1}, \delta_{n+1}, r_{n+1}$ at the next stage are determined successively by explicit algebraic inequalities.
\end{remark}

\par\medskip
\noindent\begin{minipage}{\linewidth}
\refstepcounter{figure}\label{fig:recursion}
\hypertarget{target:recursion}{}
\centering
\setlength{\unitlength}{1mm}
\footnotesize
\scalebox{0.90}{%
\begin{picture}(150,88)
\linethickness{0.4pt}
\put(0,77){\framebox(68,10){\shortstack[c]{
  Rotation number and its bound $(\alpha,c)$\\
  $L_N\ge\Gamma_c\sqrt{\ln N}$}}}
\put(72,77){\framebox(68,10){\shortstack[c]{
  Current map $F_{n-1}=H_{n-1}R_\alpha H_{n-1}^{-1}$\\
  $\mathfrak q_n\downarrow1$, $\gamma_n\downarrow1/2$, previous radii $\{r_j\}_{j<n}$}}}
\put(12,61){\framebox(116,12){\shortstack[c]{
  Determine $T_n$ and $S_n$ from $\kappa_{n-1}$\\
  Choose $N_n$ with $\ln N_n>S_n$ so that $\mathcal Z_{N_n,T_n}\le2$}}}
\put(0,43){\framebox(68,14){\shortstack[c]{
  $G_n=H_{n-1}g_{N_n,T_n}H_{n-1}^{-1}$\\
  $\|DG_n^{\ell_n}\|_\infty>4n$\\
  $K_n(t_n)/t_n>4n$,\quad $0<t_n<2^{-n}$}}}
\put(72,43){\framebox(68,14){\shortstack[c]{
  Choose amplitude $\lambda_n$ and degree $\widehat N_n$\\
  $\rho(F_n)=\alpha,\quad M_n=\lfloor\log_2\widehat N_n\rfloor$\\
  $\mathfrak E_{M_n}(F_n)>3n,\quad \mathfrak S_n(F_n)/t_n>3n$}}}
\put(12,27){\framebox(116,12){\shortstack[c]{
  Choose $r_n$ to preserve the three lower bounds $>2n$ in $B_{d_2}(F_n,r_n)$\\
  Analytically linearize $F_n$ to obtain $H_n$}}}
\put(12,0){\framebox(116,22){\shortstack[c]{
  $F_n\xrightarrow{C^2} F,\quad \rho(F)=\alpha,\quad d_2(F,R_\alpha)<\varepsilon$\\
  $F''\in\left(\bigcap_{q>1}B^0_{2,q}\right)\setminus B^0_{2,1},\quad
  \mathcal D_\gamma(F'')<\infty\ (1/2<\gamma\le1)$\\
  $\sup_\ell\|DF^\ell\|_\infty=\infty \Longrightarrow F$ and $R_\alpha$ are not $C^1$ conjugate\\
  $h(t_n)/t_n>2n\ (t_n\to 0) \Longrightarrow h\notin\mathrm{Lip}(\T),\ h'\notin L^\infty(\T)$}}}
\put(34,77){\vector(0,-1){4}}
\put(106,77){\vector(0,-1){4}}
\put(34,61){\vector(0,-1){4}}
\put(68,50){\vector(1,0){4}}
\put(106,43){\vector(0,-1){4}}
\put(70,27){\vector(0,-1){5}}
\put(128,33){\line(1,0){19}}
\put(147,33){\line(0,1){49}}
\put(147,82){\vector(-1,0){7}}
\put(145,68){\makebox(0,0)[r]{\scriptsize Next stage}}
\end{picture}%
}
\par\smallskip
{\footnotesize\textsc{Figure \thefigure.} Flowchart of the recursive construction.\par}
\end{minipage}
\par\medskip

\section{Density and meagreness of the family of counterexamples}
\label{sec:density-category}

\begin{proof}[Proof of Corollary~\ref{cor:dense-meagre}]
We first show that the defining properties of $\mathcal E_\alpha$ are preserved under smooth conjugation, and then prove density and meagreness separately.

\medskip\noindent\textbf{Invariance under smooth conjugation.}
Consider the following direct consequence of the Fourier block estimates. Let $Q\in\mathcal D^\infty(\T)$ and let $a\in C^1(\T)$ be real-valued. Define
\[
 Tw\defeq a(w\circ Q).
\]
For every continuous periodic function $w$ for which the right-hand side is finite, we have
\[
 \mathcal B_q(Tw)\le C\mathcal B_q(w)
 \qquad (1\le q\le\infty).
\]
To prove this estimate, extend the notation for sharp frequency blocks by setting $\Pi_{-1}w\defeq\widehat w(0)$, and write $H\defeq Q^{-1}$. The $L^2$ adjoint is
\[
 T^*v=(a\circ H)H'(v\circ H),
 \qquad \int_\T (Tw)\overline v\,dx
        =\int_\T w\,\overline{T^*v}\,dx.
\]
Both $T$ and $T^*$ are bounded on $W^{1,2}(\T)$. If $j\ge k$ and $j\ge0$, then
\[
 \|\Pi_jT\Pi_kw\|_2
 \le(2\pi2^j)^{-1}\|(T\Pi_kw)'\|_2
 \le C2^{k-j}\|\Pi_kw\|_2.
\]
Applying the same estimate to $T^*$ and using $L^2$ duality gives the corresponding bound for $k>j$; the case $j=k=-1$ follows directly from $L^2$ boundedness. Thus,
\begin{equation}\label{eq:smooth-conjugation-blocks}
 \|\Pi_jT\Pi_k\|_{L^2\to L^2}
 \le C2^{-|j-k|}\qquad (j,k\ge-1).
\end{equation}
For a trigonometric polynomial $w$, the discrete convolution inequality gives
\[
 \bigl\|(\|\Pi_jTw\|_2)_{j\ge-1}\bigr\|_{\ell^q}
 \le C\left(\sum_{r\in\Z}2^{-|r|}\right)
        \bigl\|(\|\Pi_kw\|_2)_{k\ge-1}\bigr\|_{\ell^q}.
\]
For a general continuous function $w$, approximation by Fourier truncations and Fatou's lemma yield the same bound. Together with the norm equivalence in Proposition~\ref{appB:block-calculus}, this proves the Besov bound above. If $\min a>0$, the inverse operator $T^{-1}v=(v/a)\circ H$ has the same form and satisfies analogous estimates.

Fix $H\in\mathcal D^\infty(\T)$ and $f\in\mathcal F_\alpha^2$, and write $Q\defeq H^{-1}$ and $G\defeq HfH^{-1}$. The second derivative satisfies
\begin{equation}\label{eq:smooth-conjugation-curvature}
 G''=a(f''\circ Q)+b,
\end{equation}
where
\begin{align*}
 a&\defeq(H'\circ f\circ Q)(Q')^2,\\
 b&\defeq(H''\circ f\circ Q)(f'\circ Q)^2(Q')^2
        +(H'\circ f\circ Q)(f'\circ Q)Q''.
\end{align*}
Here $a,b\in C^1(\T)$ and $\min a>0$. Bernstein's inequality gives
\[
 \sum_{j\ge0}\|\Pi_jb\|_2
 \le\frac{\|b'\|_2}{2\pi}\sum_{j\ge0}2^{-j}<\infty,
\]
so $b\in B^0_{2,1}(\T)$. Applying the estimates for $T$ and $T^{-1}$ to~\eqref{eq:smooth-conjugation-curvature} yields
\[
 G''\in B^0_{2,q}(\T)\iff
 f''\in B^0_{2,q}(\T)\qquad (1\le q\le\infty).
\]
In particular, nonmembership at the endpoint $q=1$ is preserved under this transformation.

For the weak Dini condition, set $L\defeq\max\{1,\|Q'\|_\infty\}$. For a continuous function $w$, the modulus of continuity satisfies
\[
 \omega_{a(w\circ Q)}(t)
 \le\|a\|_\infty\omega_w(Lt)
        +\|a'\|_\infty\|w\|_\infty t.
\]
Split the integral at $t=e^{-1}/L$. In the first part, make the change of variables $s=Lt$ and use $\ln(L/s)\ge\ln(1/s)$ to obtain, for every $0\le\gamma\le1$,
\[
 \mathcal D_\gamma\bigl(a(w\circ Q)\bigr)
 \le C_{a,Q,\gamma}
        \bigl(\mathcal D_\gamma(w)+\|w\|_\infty\bigr).
\]
The second part is trivially bounded by $C\|w\|_\infty$. Since $\omega_b(t)\le\|b'\|_\infty t$, the term $b$ in~\eqref{eq:smooth-conjugation-curvature} satisfies, for every such $\gamma$, the condition $\mathcal D_\gamma(b)<\infty$.

By Proposition~\ref{prop:conjugacy-regularity}, the conjugacies are all absolutely continuous. Let $\theta\in\R$ be the unique real number satisfying $H(h_f(\theta))=0$. By the normalization convention,
\[
 h_G(x)=H(h_f(x+\theta)),\qquad
 h_G'(x)=H'(h_f(x+\theta))h_f'(x+\theta)
 \qquad\text{almost everywhere}.
\]
Since the positive periodic function $H'$ is bounded above and bounded away from zero, $h_G'\notin L^\infty(\T)\iff h_f'\notin L^\infty(\T)$. By Proposition~\ref{lem:rotation-basic}, the rotation number is invariant under conjugation. Hence $HfH^{-1}\in\mathcal E_\alpha$ whenever $f\in\mathcal E_\alpha$.

\medskip\noindent\textbf{Density.}
The subclass $\mathcal F_\alpha^\omega$ is dense in $\mathcal F_\alpha^2$ in the $C^2$ topology. Indeed, for any $F\in\mathcal F_\alpha^2$, approximate in the $C^2$ norm the periodic function $F-\id$ by its Fej\'er means, and denote the corresponding analytic lifts by $A_n$. Since $\min F'>0$, one has $A_n'>0$ for all sufficiently large $n$. Set $\eta_n\defeq\|A_n-F\|_\infty$. Then we have the pointwise bounds
\[
 A_n-\eta_n\le F\le A_n+\eta_n.
\]
By the monotonicity and continuity of the rotation number in Proposition~\ref{lem:rotation-basic}, there is a constant $\tau_n\in[-\eta_n,\eta_n]$ such that $\rho(A_n+\tau_n)=\alpha$. Hence $A_n+\tau_n\in\mathcal F_\alpha^\omega$ and $d_2(A_n+\tau_n,F)\to0$.

Let $F\in\mathcal F_\alpha^2$ and $\varepsilon>0$ be arbitrary. Choose $G_0\in\mathcal F_\alpha^\omega$ such that $d_2(G_0,F)<\varepsilon/2$. Since $\alpha$ is of constant type, the Herman--Yoccoz global analytic linearization theorem \cite[Corollary, p.~335]{Yoc84} gives $H\in\mathcal D^\omega(\T)$ such that $G_0=HR_\alpha H^{-1}$. Applying Theorem~\ref{thm:main} with $q_*=2$, we find, arbitrarily close to $R_\alpha$ in the metric $d_2$, a map $f\in\mathcal E_\alpha$. Specifically, choose $f$ such that
\[
 d_2(f,R_\alpha)
 <\min\left\{\frac12,
        \frac{\varepsilon}{2C_{\mathrm{conj}}(H)}\right\}.
\]
The derivative bounds required in~\eqref{eq:conjugation-stability} then hold, giving
\[
 d_2(HfH^{-1},G_0)<\frac{\varepsilon}{2}.
\]
Conjugation invariance implies that $HfH^{-1}\in\mathcal E_\alpha$, and
\[
 d_2(HfH^{-1},F) \le d_2(HfH^{-1},G_0) + d_2(G_0,F) < \varepsilon.
\]

\medskip\noindent\textbf{Meagreness.}
For $M\in\N$, define
\[
 \mathcal A_M\defeq
 \{F\in\mathcal F_\alpha^2:\mathcal D_1(F'')\le M\}.
\]
This set is closed in $\mathcal F_\alpha^2$. Indeed, for continuous periodic functions $u,v$, we have
\[
 |\omega_u(t)-\omega_v(t)|\le2\|u-v\|_\infty,
\]
so, if $F_n\to F$ in the $C^2$ topology, Fatou's lemma gives $\mathcal D_1(F'')\le\liminf_{n\to\infty}\mathcal D_1(F_n'')$.

To show that $\mathcal A_M$ has empty relative interior, fix $F\in\mathcal F_\alpha^2$ and $\varepsilon>0$, and choose $0<\delta<\varepsilon$. For a sufficiently large integer $N$, set
\[
 \widetilde F_N(x)\defeq
 F(x)-\frac{\delta}{(2\pi N)^2}\sin(2\pi Nx).
\]
The derivative of this lift is strictly positive for sufficiently large $N$. As in the rotation-number correction above, choose a constant satisfying $|\tau_N|\le\delta/(2\pi N)^2$ such that $F_N\defeq\widetilde F_N+\tau_N\in\mathcal F_\alpha^2$. Then
\[
 d_2(F_N,F)\le\delta<\varepsilon,\qquad
 F_N''=F''+\delta\sin(2\pi Nx).
\]
At $x=1/(4N)$ and $y=3/(4N)$, the sine function takes the values $1$ and $-1$, respectively. Since $|x-y|=1/(2N)$, for all sufficiently large $N$ we have
\[
 \omega_{F_N''}\bigl((2N)^{-1}\bigr)
 \ge2\delta-\omega_{F''}\bigl((2N)^{-1}\bigr)
 \ge\delta.
\]
Consequently,
\[
 \mathcal D_1(F_N'')
 \ge\delta\int_{(2N)^{-1}}^{e^{-1}}
        \frac{dt}{t\ln(1/t)}
 =\delta\ln\ln(2N)\longrightarrow\infty.
\]
Thus every relative neighborhood of $F$ contains a point outside $\mathcal A_M$, so each $\mathcal A_M$ is closed and nowhere dense. Finally, the weak Dini condition in the definition of $\mathcal E_\alpha$ gives
\[
 \mathcal E_\alpha\subset\bigcup_{M=1}^{\infty}\mathcal A_M.\qedhere
\]
\end{proof}

\section{Global rigidity at the Besov endpoint}\label{sec:rigidity}

This section proves Theorem~\ref{thm:rigidity}. The argument establishes a scale-by-scale comparison between Fourier blocks and averages over dynamical orbit partitions, and combines this comparison with the Katznelson--Ornstein return distortion estimates to obtain absolute summability; the required dynamical results are collected in Appendix~\ref{app:KO}. Throughout this section, the constant $C$ may depend on the map $F$ and its rotation number $\alpha$.

\subsection{Multiplication by \texorpdfstring{$C^1$}{C1} functions}
\label{subsec:dynamical-averages}

To deduce from $F''\in B^0_{2,1}(\T)$ that $F''/F'\in B^0_{2,1}(\T)$, we establish the following estimate.
\begin{lemma}\label{lem:besov-C1-multiplier}
If $a\in C^1(\T)$ and $f\in B^0_{2,1}(\T)$, then
\[
 \|af\|_{B^0_{2,1}}
 \le C\bigl(\|a\|_\infty+\|a'\|_\infty\bigr)\|f\|_{B^0_{2,1}}.
\]
\end{lemma}
\begin{proof}
Write $f_j\defeq\mathsf P_jf$. By Definition~\ref{def:besov}, $\sum_{j\ge0}\|f_j\|_2=\|f\|_{B^0_{2,1}}<\infty$, so $\sum_{j\ge0}f_j$ converges absolutely in $L^2(\T)$ to $f$, and
\[
 \|f\|_2\le\sum_{j\ge0}\|f_j\|_2=\|f\|_{B^0_{2,1}}.
\]
For $j\ge1$, write $\psi(\xi)\defeq\chi(\xi)-\chi(2\xi)$ and
\[
 b_{j,k}\defeq\psi(2^{-j}k),\qquad
 \mathcal K_j(y)\defeq\sum_{k\in\Z}b_{j,k}e^{2\pi iky}.
\]
By the definition of the Fourier multiplier, every $v\in L^2(\T)$ admits the convolution representation $\mathsf P_jv(x)=\int_\T\mathcal K_j(y)v(x-y)\,dy$.

Define $\Delta b_{j,k}\defeq b_{j,k}-b_{j,k-1}$. Three successive termwise summations give
\[
 (1-e^{2\pi iy})^3\mathcal K_j(y)
 =\sum_{k\in\Z}\Delta^3b_{j,k}e^{2\pi iky},
\]
where the third difference is represented by
\[
 \Delta^3b_{j,k}
 =2^{-3j}\int_{[0,1]^3}
 \psi'''\bigl(2^{-j}(k-t_1-t_2-t_3)\bigr)
 \,dt_1\,dt_2\,dt_3.
\]
Since $\supp\psi\subset[-2,2]$, the sum has at most $C2^j$ nonzero terms, giving $\sum_k|\Delta^3b_{j,k}|\le C2^{-2j}$ and $\sum_k|b_{j,k}|\le C2^j$. For $t\defeq\dist(y,\Z)\in(0,1/2]$, the inequality $|1-e^{2\pi iy}|=2\sin(\pi t)\ge4t$ yields
\[
 |\mathcal K_j(y)|
 \le C\min\{2^j,\,2^{-2j}t^{-3}\}.
\]
Splitting the integral over $[-1/2,1/2]$ gives
\begin{equation}\label{eq:kernel-first-moment}
 \int_\T\dist(y,\Z)|\mathcal K_j(y)|\,dy
 \le 2\left(2^j\int_0^{2^{-j}}t\,dt+2^{-2j}\int_{2^{-j}}^{1/2}t^{-2}\,dt\right)
 \le C2^{-j}.
\end{equation}

Since $a\in C^1(\T)$, we have $|a(x-y)-a(x)|\le\|a'\|_\infty\dist(y,\Z)$. By the convolution formula,
\[
 \mathsf P_j(af)(x)-a(x)f_j(x)
 =\int_\T\mathcal K_j(y)\bigl(a(x-y)-a(x)\bigr)f(x-y)\,dy.
\]
Together with~\eqref{eq:kernel-first-moment}, this gives, for every $j\ge1$,
\[
 \|\mathsf P_j(af)-af_j\|_2
 \le \|a'\|_\infty\|f\|_2\int_\T\dist(y,\Z)|\mathcal K_j(y)|\,dy
 \le C2^{-j}\|a'\|_\infty\|f\|_2.
\]
Applying the triangle inequality in $L^2$, we obtain
\[
 \|\mathsf P_j(af)\|_2
 \le\|a\|_\infty\|f_j\|_2+C2^{-j}\|a'\|_\infty\|f\|_2\qquad(j\ge1).
\]
For $j=0$, Plancherel's identity gives $\|\mathsf P_0(af)\|_2\le\|af\|_2\le\|a\|_\infty\|f\|_2$. Summing over all $j\ge0$ and substituting $\|f\|_2\le\|f\|_{B^0_{2,1}}$ gives
\[
 \|af\|_{B^0_{2,1}}
 =\sum_{j\ge0}\|\mathsf P_j(af)\|_2
 \le C\bigl(\|a\|_\infty+\|a'\|_\infty\bigr)\|f\|_{B^0_{2,1}}.\qedhere
\]
\end{proof}

\subsection{Proof of Theorem~\ref{thm:rigidity}}
\label{subsec:global-linearization}

\subsubsection{Dynamical partitions and measures of atoms}
\label{subsubsec:orbit-partitions}
Let $F\in\mathcal F_\alpha^2$ have a rotation number $\alpha$ of constant type, and suppose that $F''\in B^0_{2,1}(\T)$. Since $\min_\T F'>0$, Lemma~\ref{lem:besov-C1-multiplier} gives
\[
 g\defeq F''/F'\in B^0_{2,1}(\T).
\]
Since $g=D\ln F'$ and $F'$ is periodic, we have $\int_\T g=0$. By Denjoy's theorem, there is a unique normalized topological conjugacy $h$ satisfying~\eqref{eq:forward-conjugacy}; Proposition~\ref{appC:KO-interval-ratios} asserts that $h$ is absolutely continuous and $h'\in L^2(\T)$. Write
\[
 H\defeq\|h'\|_2^2.
\]
The Cauchy--Schwarz inequality gives
\[
 H\ge\left(\int_\T h'\right)^2=1.
\]

Let $p_n/q_n$ be the sequence of continued-fraction convergents of $\alpha$. For a measurable set $E\subset\T$, write $|E|$ for its Lebesgue measure. For the rigid rotation $R_\alpha$, denote the circle partition determined by the orbit points $\{k\alpha\pmod1: 0\le k<q_n\}$ by $P_n(R_\alpha)$, and denote its open interval components by $J$. Since the points $\{kp_n/q_n\pmod1: 0\le k<q_n\}$ form a uniform grid of spacing $q_n^{-1}$ and $|\alpha-p_n/q_n|<(q_nq_{n+1})^{-1}$, we obtain
\begin{equation}\label{eq:rotation-mesh-elementary}
 |J|\le q_n^{-1}+2q_{n+1}^{-1}\le\frac3{q_n}.
\end{equation}

Let $P_n(F)$ denote the dynamical partition generated by the orbit points $\{F^k(0)\pmod1: 0\le k<q_n\}$. Its interval components are called atoms and denoted by $I$; its mesh is defined by
\begin{equation}\label{eq:mesh-definition}
 \operatorname{mesh}P_n(F)\defeq\max_{I\in P_n(F)}|I|.
\end{equation}
For an integrable function $f$, define
\begin{equation}\label{eq:dynamical-average}
 (E_nf)|_I\defeq\frac1{|I|}\int_I f(x)\,dx\qquad(I\in P_n(F)),
\end{equation}
and set $E_0f\defeq\int_\T f$. By the Cauchy--Schwarz inequality, each atom satisfies
\[
 |I|\,|(E_nf)|_I|^2\le\int_I|f|^2,
\]
so, for the partition $P_n(F)$ and any union of its atoms $B\subset\T$, summation gives $\|\mathbf1_BE_nf\|_2\le\|\mathbf1_Bf\|_2$.

By the conjugacy relation, every atom has the form $I=h(J)$. In view of the absolute continuity of $h$ and the Cauchy--Schwarz inequality, we always have
\[
 |I|^2=\left(\int_Jh'(x)\,dx\right)^2\le |J|\int_J|h'(x)|^2\,dx.
\]
Thus, for every scale $L>0$,
\begin{equation}\label{eq:short-atom-measure}
 \left|\bigcup_{\substack{I\in P_n(F)\\ |I|<L}}I\right|\le q_nL,\qquad
 \left|\bigcup_{\substack{I\in P_n(F)\\ |I|\ge L}}I\right|\le L^{-1}\sum_{I\in P_n(F)}|I|^2\le\frac{3H}{q_nL}.
\end{equation}

\subsubsection{A comparison principle for Fourier blocks and dynamical partitions}
\label{subsubsec:comparison-lemma}

\begin{lemma}\label{lem:Fourier-orbit-comparison}
There is a constant $C_H$, depending only on the constant-type bound and $H$, such that every $f\in B^0_{2,1}(\T)$ with $\int_\T f=0$ satisfies
\begin{equation}\label{eq:Fourier-orbit-comparison}
 \sum_{n\ge1}\|(E_n-E_{n-1})f\|_{3/2}
 \le C_H\|f\|_{B^0_{2,1}}.
\end{equation}
\end{lemma}
\begin{proof}
For an integer $N\ge1$, let $f_N$ be a trigonometric polynomial with Fourier support in $\{N\le|k|<2N\}$, and denote its zero-mean periodic primitive by $u_N$. Then
\begin{equation}\label{eq:block-derivative-primitive}
 \|f_N'\|_2\le4\pi N\|f_N\|_2,\qquad
 \|u_N\|_2\le(2\pi N)^{-1}\|f_N\|_2.
\end{equation}
Fix $n\ge1$, write $Q\defeq q_n$, and choose $L\defeq(NQ)^{-1/2}$. Set $B\defeq\bigcup_{I\in P_n(F),\,|I|<L}I$. For every measurable set $E\subset\T$, we have $\|\mathbf1_Ew\|_{3/2}\le |E|^{1/6}\|\mathbf1_Ew\|_2$.

Suppose that $Q\le N$. By~\eqref{eq:short-atom-measure},
\[
 \|\mathbf1_BE_nf_N\|_{3/2}\le(QL)^{1/6}\|f_N\|_2.
\]
For a long atom satisfying $|I|\ge L$, write $I=[a,b]$ and take the intervals contained in $I$ adjacent to its endpoints, namely $[a,a+\delta]$ and $[b-\delta,b]$, where $\delta\defeq(2N)^{-1}\le L/2$. Averaging $u_N(a)=u_N(x)-\int_a^x f_N$ over $x\in[a,a+\delta]$ and applying the Cauchy--Schwarz inequality gives
\[
 |u_N(a)|^2\le\frac2\delta\int_a^{a+\delta}|u_N|^2+2\delta\int_a^{a+\delta}|f_N|^2.
\]
Since the endpoint intervals of distinct long atoms have disjoint interiors, summing over all long atoms and using~\eqref{eq:block-derivative-primitive} gives
\[
 \sum_{|I|\ge L}\bigl(|u_N(a)|^2+|u_N(b)|^2\bigr)
 \le C\bigl(N\|u_N\|_2^2+N^{-1}\|f_N\|_2^2\bigr)
 \le CN^{-1}\|f_N\|_2^2.
\]
Since $(E_nf_N)|_I=(u_N(b)-u_N(a))/|I|$, we have
\[
 \|\mathbf1_{B^c}E_nf_N\|_2^2
 =\sum_{|I|\ge L}\frac{|u_N(b)-u_N(a)|^2}{|I|}
 \le\frac{C}{NL}\|f_N\|_2^2.
\]
Using $\|w\|_{3/2}\le\|w\|_2$ and $L=(NQ)^{-1/2}$, we obtain
\begin{equation}\label{eq:coarse-average-block}
 \|E_nf_N\|_{3/2}
 \le C\bigl((QL)^{1/6}+(NL)^{-1/2}\bigr)\|f_N\|_2
 \le C(Q/N)^{1/12}\|f_N\|_2\qquad(Q\le N).
\end{equation}

Now suppose that $Q\ge N$. For each interval $I$, write $(f_N)_I\defeq(E_nf_N)|_I$. The fundamental theorem of calculus gives
\[
\int_I|f_N-(E_nf_N)|_I|^2 \le |I|^2\int_I|f_N'|^2.
\]
This follows directly by writing $f_N(x)-(E_nf_N)|_I=|I|^{-1}\int_I(f_N(x)-f_N(y))\,dy$ and applying the Cauchy--Schwarz inequality twice. Summing over the atoms in $B$ gives
\[
 \|\mathbf1_B(f_N-E_nf_N)\|_2 \le L\|f_N'\|_2 \le CNL\|f_N\|_2.
\]
On $B^c$,~\eqref{eq:short-atom-measure} yields
\[
 \|\mathbf1_{B^c}(f_N-E_nf_N)\|_{3/2}\le 2|B^c|^{1/6}\|f_N\|_2\le 2\left(\frac{3H}{QL}\right)^{1/6}\|f_N\|_2.
\]
Substituting $L=(NQ)^{-1/2}$, we obtain
\begin{equation}\label{eq:fine-average-block}
 \|f_N-E_nf_N\|_{3/2}
 \le C\left(NL+\left(\frac{H}{QL}\right)^{1/6}\right)\|f_N\|_2
 \le C_H(N/Q)^{1/12}\|f_N\|_2\qquad(Q\ge N).
\end{equation}

Set
\[
 n_N\defeq\max\{n\ge0:q_n\le N\}.
\]
This index is well defined because $q_0=1\le N$ and $q_n\to\infty$, and it satisfies $q_{n_N}\le N<q_{n_N+1}$. Since $E_0f_N=\int_\T f_N=0$, for $1\le n\le n_N$ we apply~\eqref{eq:coarse-average-block} to $E_n$ and, when $n\ge2$, to $E_{n-1}$. This range is empty if $n_N=0$; for $n=n_N+1$, we have $\|(E_{n_N+1}-E_{n_N})f_N\|_{3/2}\le2\|f_N\|_2$; for $n\ge n_N+2$, apply~\eqref{eq:fine-average-block} to $(E_n-E_{n-1})f_N=(E_nf_N-f_N)-(E_{n-1}f_N-f_N)$. Since $q_{n+2}\ge2q_n$, we have
\[
 \sum_{1\le n\le n_N}(q_n/N)^{1/12}+\sum_{n\ge n_N+2}(N/q_{n-1})^{1/12}\le C,
\]
and hence
\begin{equation}\label{eq:single-block-sum}
 \sum_{n\ge1}\|(E_n-E_{n-1})f_N\|_{3/2}\le C_H\|f_N\|_2.
\end{equation}
For every zero-mean function $f\in B^0_{2,1}(\T)$, the dyadic expansion $f=\sum_{k\ge0}\Pi_kf$ converges absolutely in $L^2(\T)$. Applying~\eqref{eq:single-block-sum} and Proposition~\ref{appB:block-calculus} gives
\[
 \sum_{n\ge1}\|(E_n-E_{n-1})f\|_{3/2}
 \le C_H\sum_{k\ge0}\|\Pi_kf\|_2
 \le C_H\|f\|_{B^0_{2,1}}.\qedhere
\]
\end{proof}

\subsubsection{Summability estimates for return distortion}
\label{subsubsec:return-sum}
For $g\defeq F''/F'\in B^0_{2,1}(\T)$, define, as in~\eqref{eq:dynamical-average},
\begin{equation}\label{eq:KO-averaging}
 \mathsf G_n\defeq E_ng\quad(n\ge0),\qquad
 \mathsf g_n\defeq\mathsf G_n-\mathsf G_{n-1}\quad(n\ge1).
\end{equation}
Each $\mathsf G_n$ is constant on the atoms of $P_n(F)$, and $\int_J \mathsf g_n\,dx=0$ for every $J\in P_{n-1}(F)$. By Lemma~\ref{lem:Fourier-orbit-comparison},
\begin{equation}\label{eq:summable-averaging-hypothesis}
 \sum_{n\ge1}\|\mathsf g_n\|_{3/2}
 \le C_H\|g\|_{B^0_{2,1}}<\infty.
\end{equation}
By~\eqref{eq:rotation-mesh-elementary} and the uniform continuity of $h$, we have $\lim_{n\to\infty}\operatorname{mesh}P_n(F)=0$, and hence
\[
 \|\mathsf G_n-g\|_\infty
 \le\omega_g(\operatorname{mesh}P_n(F))\longrightarrow0 \qquad(n\to\infty).
\]
Since $\mathsf G_0=\int_\T g\,dx=0$, the series $\sum_{n\ge1}\mathsf g_n$ converges uniformly to $g$.

If an open arc $I$ has its first $q_n$ iterates $F^k(I)$ pairwise disjoint for $0\le k<q_n$, then $I$ is called a $q_n$-small interval. In addition to $K_n^0$ in~\eqref{eq:KO-return-distortion}, introduce
\begin{equation}\label{eq:KO-short-distortion}
 \widetilde K_n^1\defeq
 \sup_{\substack{0\le k<q_n\\ I=(x,y)\ q_n\text{-small}}}
 |\ln DF^k(y)-\ln DF^k(x)|.
\end{equation}
Since $F\in\mathcal F_\alpha^2$, for every fixed $n$ and $0\le k<q_n$, we have $DF^k>0$ and this derivative is uniformly continuous on $\T$, so $\widetilde K_n^1<\infty$. By Proposition~\ref{appC:KO-smallness}, for $n\ge2$ we have $K_n^0\le2\widetilde K_n^1$.

Take $p=3/2$ in Proposition~\ref{appC:KO-recursion}. Write
\[
 s_n\defeq \mathsf r_n+C\sum_{j\ge1}c_{nj}\|\mathsf g_j\|_{3/2} \qquad (n\ge n_0).
\]
By~\eqref{appC:KO-recursion-remainder},
\[
 \sup_{j\ge1}\sum_{n\ge n_0}c_{nj}
 \le\sum_{n\ge1}n\eta^{n/2}
    +\sum_{k\in\Z}\eta_1^{|k|}<\infty.
\]
Combining this with~\eqref{eq:summable-averaging-hypothesis} and interchanging the order of summation gives $\sum_{n\ge n_0}s_n<\infty$. Choose $\varepsilon>0$ sufficiently small that
\[
 \kappa\defeq \frac{C\varepsilon\eta}{1-\eta}<\frac{1}{2}.
\]
For the above choice of $\varepsilon$, fix the corresponding $n_0$. Summing~\eqref{appC:KO-recursion-estimate} over $n_0\le n\le N$ and applying Fubini's theorem gives
\[
 \sum_{n=n_0}^N\widetilde K_n^1
 \le C_1+\kappa\sum_{m=n_0}^N\widetilde K_m^1
    +\sum_{n\ge n_0}s_n,
\]
where the constant $C_1$ depends only on the finite initial segment $\widetilde K_1^1,\dots,\widetilde K_{n_0-1}^1$. Since $\kappa<1/2$, for every $N\ge n_0$ the preceding inequality gives
\[
 \sum_{n=n_0}^N\widetilde K_n^1
 \le \frac{C_1+\sum_{n\ge n_0}s_n}{1-\kappa}.
\]
Letting $N\to\infty$, we obtain
\begin{equation}\label{eq:global-summable-returns}
 \sum_{n\ge1}\widetilde K_n^1<\infty,
 \qquad\sum_{n\ge1}K_n^0<\infty.
\end{equation}

\begin{remark}\label{rem:C1-criterion-connection}
Once $\sum_{n\ge1}K_n^0<\infty$ is established in~\eqref{eq:global-summable-returns}, the conclusion $h\in\mathcal D^1(\T)$ also follows from Proposition~\ref{appC:KO-C1-criterion}. In Subsection~\ref{subsubsec:continuous-derivative}, we prove this conclusion using the interval-ratio estimate in Proposition~\ref{appC:KO-interval-ratios}, which gives uniform convergence of the dyadic difference quotients to a strictly positive continuous function.
\end{remark}

\subsubsection{Continuity and strict positivity of the derivative}
\label{subsubsec:continuous-derivative}
Let the indices $n_*,\ell(n)$ and the constants $\theta,L_A$ be as in Proposition~\ref{appC:KO-interval-ratios}, and use the sequence $(\beta_n)_{n\ge1}$ from Lemma~\ref{lem:KO-dyadic}. Its tail satisfies
\[
 \sum_{n\ge n_*}\beta_n
 \le C\sum_{m\ge1}K_m^0
       \sum_{\substack{n\ge n_*\\\ell(n)\le m}}
          \theta^{m-\ell(n)}
 \le\frac{CL_A}{1-\theta}\sum_{m\ge1}K_m^0<\infty.
\]
Together with~\eqref{eq:global-summable-returns}, this yields $\sum_{n\ge1}\beta_n<\infty$. Define
\[
 v_n(x)\defeq 2^n\bigl(h(x+2^{-n})-h(x)\bigr) \qquad (n\ge1).
\]
Set $\delta\defeq 2^{-n-1}$ and $\xi_n(x)\defeq v_{n+1}(x+\delta)/v_{n+1}(x)$. By the interval-ratio estimate in Proposition~\ref{appC:KO-interval-ratios}, we have $|\xi_n(x)-1|\le\beta_{n+1}$, and
\[
 v_n(x)=\frac{1}{2}v_{n+1}(x)\bigl(1+\xi_n(x)\bigr).
\]
Since $\beta_n\to0$, for all sufficiently large $n$ we have
\[
 \frac{2}{2+\beta_{n+1}}
 \le\frac{v_{n+1}(x)}{v_n(x)}
 \le\frac{2}{2-\beta_{n+1}},
 \qquad
 \left|\frac{v_{n+1}(x)}{v_n(x)}-1\right|\le\beta_{n+1}\le\frac{1}{2}.
\]
Since $\sum_{n\ge1}\beta_n<\infty$, the infinite product $\prod_{n\ge1}(v_{n+1}/v_n)$ converges uniformly on $\T$. Thus there are constants $0<c_F\le C_F<\infty$ such that $c_F\le v_n(x)\le C_F$ for all sufficiently large $n$. It follows that
\[
 \|v_{n+1}-v_n\|_\infty\le C_F\beta_{n+1}.
\]
Hence $(v_n)_{n\ge1}$ is a Cauchy sequence in $C(\T)$ and converges uniformly to a continuous limit $v$ satisfying $v\ge c_F$.

Fix $x\in\R$. Changing variables in the two integrals and canceling the common terms gives
\[
 \int_0^x v_n(s)\,ds
 =2^n\int_0^{2^{-n}}\bigl(h(x+t)-h(t)\bigr)\,dt.
\]
Since $h$ is continuous and $v_n\to v$ uniformly, letting $n\to\infty$ gives
\[
 \int_0^x v(s)\,ds = h(x)-h(0).
\]
Thus $h$ is differentiable everywhere and $h'=v\in C(\T)$. Since $h'\ge c_F>0$, the map $h$ is a $C^1$ diffeomorphism, that is, $h\in\mathcal D^1(\T)$.

Finally, by~\eqref{appB:lifting}, $F''\in B^0_{2,1}(\T)$ is equivalent to $F-R_\alpha\in B^2_{2,1}(\T)$.\qed

\begin{remark}
\label{rem:rigidity-methods}
Proposition~\ref{appC:KO-C1-criterion} summarizes two conclusions from~\cite{KO89D}: for a rotation number of constant type, absolute summability of $K_n^0$ implies $C^1$ conjugacy; if $F''\in C^\sigma(\T)$, $0<\sigma<1$, then $K_n^0$ decays exponentially. The former is established by decomposing circle translations into continued-fraction returns and controlling the derivative products arising from the chain rule, whereas the latter relies on a recursion for nonlinear distortion on return intervals. At the critical $C^2$ regularity, Proposition~\ref{appC:KO-smallness} gives only $\widetilde K_n^1\to0$, which does not a priori ensure $\sum_{n\ge1}K_n^0<\infty$.

The absolute continuity argument in~\cite{KO89} requires only that $\ln F'$ be absolutely continuous and $D\ln F'\in L^p(\T)$ for $p>1$; Proposition~\ref{appC:KO-square-summability} states the square summability of distortion that this argument yields for $C^2$ maps. Its main step is to take martingale-type averages of the logarithmic derivative along nested orbit partitions and control them using the convolution-type recursion in Proposition~\ref{appC:KO-recursion}. In particular, for rotation numbers of constant type, Proposition~\ref{appC:KO-interval-ratios} ensures that $h$ and $h^{-1}$ have $L^2$ derivatives.

Tong, Xiao, and Li~\cite{TXL} use cross-ratio distortion estimates controlled by the modulus of continuity of $F''$ to establish Denjoy-type bounds for $DF^{q_n}-1$. Absolute summation of the return errors yields a continuous positive density for the invariant measure and hence the inverse conjugacy; see~\cite[Proposition 3.2, Theorems 4.1--4.2]{TXL}. In the constant-type case, their sufficient condition is $\mathcal D_0(F'')<\infty$; see~\cite[Corollary 4.1, Remark 4.4]{TXL}. Summing the modulus estimate in~\cite[Eq.~4.7]{TXL} over the characteristic dynamical scales naturally corresponds to the classical Dini condition. Indeed, for any fixed $0<\lambda<1$, Tonelli's theorem gives
\[
 \sum_{n\ge0}\lambda^n\int_{\lambda^n}^{1}
        \frac{\omega_{F''}(t)}{t^2}\,dt
 =\int_0^1\frac{\omega_{F''}(t)}{t^2}
        \sum_{\{n:\,\lambda^n\le t\}}\lambda^n\,dt
 \asymp\int_0^1\frac{\omega_{F''}(t)}t\,dt.
\]
After establishing $C^1$ conjugacy, Tong, Xiao, and Li~\cite[Theorem 4.4, Corollary 4.2]{TXL} use the tails of the return-error sums to obtain quantitative bounds for the modulus of continuity of the derivative of the conjugacy.

The proof of Theorem~\ref{thm:rigidity} combines Besov estimates with dynamical averaging increments. Equation~\eqref{eq:Fourier-orbit-comparison} directly converts summability of the Fourier block series of the logarithmic derivative into $L^{3/2}$ summability of the averaging increments over orbit partitions, with the measure estimate for exceptional atoms contributing a regularity gain of exponent $1/6$. Substituting this $L^{3/2}$ summability into the linear convolution system in Proposition~\ref{appC:KO-recursion} and absorbing the small perturbation yields the summability criterion in Proposition~\ref{appC:KO-C1-criterion}. Finally, Proposition~\ref{appC:KO-interval-ratios} shows that the dyadic difference quotients form a Cauchy sequence in the uniform norm, establishing that $h'\in C(\T)$ and is bounded away from zero. In particular, the condition $F''\in B^0_{2,1}(\T)$ provides a purely analytic criterion for $C^1$ conjugacy; by Lemma~\ref{lem:half-dini}, this condition follows from the weak Dini condition $\mathcal D_{1/2}(F'')<\infty$.
\end{remark}

\appendix
\section{Proof of the finite-band square estimate}\label{app:square}

\begin{proof}[Proof of Lemma~\ref{lem:finite-band}]
We argue by induction on the number of nonzero dyadic blocks. The inequality is immediate when there is at most one nonzero dyadic block. Suppose that the assertion holds for every family with fewer nonzero blocks, and write
\[
 W\defeq\sum_sW_s,\qquad
 A_s\defeq\|W_s\|_4,\qquad
 \mathcal A\defeq\left(\sum_sA_s^2\right)^{1/2}.
\]
Decompose the square as $W^2=\mathcal H+2\mathcal T$, where
\[
 \mathcal H\defeq\sum_{|s-t|\le1}W_sW_t,
 \qquad
 \mathcal G_s\defeq\sum_{t\le s-2}W_t,
 \qquad
 \mathcal T\defeq\sum_sW_s\mathcal G_s.
\]
H\"older's inequality $\|W_sW_t\|_2\le A_sA_t$ gives directly
\begin{equation}\label{eq:resonant}
 \|\mathcal H\|_2
 \le\sum_s(A_s^2+A_sA_{s-1}+A_sA_{s+1})
 \le3\mathcal A^2.
\end{equation}

For $s\ge0$, write $X_s\defeq W_s\mathcal G_s$, with $X_0=X_1=0$. For $s\ge2$, since $\supp\widehat W_s\subset I_s$ and $\supp\widehat{\mathcal G_s}\subset\{|k|<2^{s-1}\}$, we have
\[
 \supp\widehat X_s
 \subset\{k\in\Z:2^{s-1}<|k|<5\cdot2^{s-1}\}.
\]
Each nonzero frequency belongs to at most three such frequency bands, as the admissible indices lie in an interval of length $\log_2 5 < 3$. At each frequency, the Cauchy--Schwarz inequality bounds the square of the sum of at most three coefficients by three times the sum of their squares. Summing over all frequencies and applying Plancherel's identity yields
\[
 \left\|\sum_sX_s\right\|_2^2\le3\sum_s\|X_s\|_2^2.
\]
For each $W_s\ne0$, the sum $\mathcal G_s$ contains strictly fewer nonzero blocks. By H\"older's inequality and the induction hypothesis,
\begin{align*}
 \|2\mathcal T\|_2^2
 &\le12\sum_sA_s^2\|\mathcal G_s\|_4^2\\
 &\le192\sum_sA_s^2\sum_{t\le s-2}A_t^2
 \le96\mathcal A^4,
\end{align*}
where the double sum satisfies
\[
 \sum_sA_s^2\sum_{t\le s-2}A_t^2
 \le\sum_{t<s}A_s^2A_t^2
 =\tfrac12\left[\left(\sum_sA_s^2\right)^2-\sum_sA_s^4\right]
 \le\tfrac12\mathcal A^4.
\]
Combining~\eqref{eq:resonant} with the preceding estimates gives
\[
 \|W\|_4^2=\|W^2\|_2
 \le\|\mathcal H\|_2+\|2\mathcal T\|_2
 \le(3+4\sqrt6)\mathcal A^2
 \le16\sum_s\|W_s\|_4^2.\qedhere
\]
\end{proof}

\section{Basic facts about function spaces}
\label{app:function-spaces}

This section collects the norm comparisons, embedding theorems, and dyadic estimates used in the main text. All function spaces are periodic and inhomogeneous, with the projections given in Definition~\ref{def:besov}. We use the periodic versions of the following lemmas; for the passage between the Euclidean and periodic function space settings, see \cite[Chapter~9, pp.~264--268]{Tri83}.

\begin{proposition}
\label{appB:block-calculus}
Let $s\in\R$ and $1\le p,q\le\infty$. The definition of the Besov norm is independent of the choice of smooth cutoff. When $p=2$,
\begin{equation}\label{appB:sharp-equivalence}
 \|w\|_{B^s_{2,q}} \asymp |\widehat w(0)| + \bigl\|(2^{sj}\|\Pi_jw\|_2)_{j\ge0}\bigr\|_{\ell^q}.
\end{equation}
For $m=1,2$, differentiation satisfies
\begin{equation}\label{appB:lifting}
 \|w\|_{B^{s+m}_{p,q}} \asymp |\widehat w(0)| + \|D^mw\|_{B^s_{p,q}}.
\end{equation}
Hence, for an absolutely continuous map $h$,
\begin{equation}\label{appB:lift-convention}
 h-\id\in B^1_{p,q} \quad\Longleftrightarrow\quad h'\in B^0_{p,q}.
\end{equation}
The summability exponents satisfy the continuous embedding
\begin{equation}\label{appB:q-monotonicity}
 B^s_{p,q_0} \hookrightarrow B^s_{p,q_1} \qquad (1\le q_0\le q_1\le\infty).
\end{equation}
\end{proposition}

For cutoff independence, monotonicity in the summability exponent, and differentiation properties, see \cite[Sections~2.3.2 and 2.3.8]{Tri83}; \eqref{appB:sharp-equivalence} is the $L^2$ decomposition associated with sharp annular projections, whose periodic version is discussed in \cite[Chapter~9]{Tri83}.

\begin{proposition}
\label{appB:embeddings}
\label{appB:besov-sobolev}
\label{appB:classical-embeddings}
\label{appB:arithmetic-versus-sobolev}
For every $1<p<\infty$,
\begin{equation}\label{appB:sobolev-sandwich}
 B^1_{p,\min\{p,2\}}\hookrightarrow W^{1,p}
 \hookrightarrow B^1_{p,\max\{p,2\}}.
\end{equation}
In particular,
\[
 \begin{array}{ll}
 1<p<2: & W^{1,p}\hookrightarrow B^1_{p,2}, \\
 p=2: & W^{1,2}=B^1_{2,2}, \\
 2<p<\infty: & B^1_{p,2}\hookrightarrow W^{1,p}.
 \end{array}
\]
At order zero, $B^0_{2,1}\hookrightarrow L^2=B^0_{2,2}$. For every integer $m\ge0$,
\begin{equation}\label{appB:critical-continuity}
 B^{m+1/2}_{2,1}\hookrightarrow C^m.
\end{equation}
In particular, $B^2_{2,1}\hookrightarrow C^1$, but
\[
 B^0_{2,1}\not\hookrightarrow C^0,
 \qquad B^2_{2,1}\not\hookrightarrow C^2.
\]
For $0<\sigma<1$ and $m=1,2$, the following identity holds with equivalent norms:
\begin{equation}\label{appB:holder-identification}
 B^{m+\sigma}_{\infty,\infty}=C^{m+\sigma}.
\end{equation}
Membership in the classical spaces above is understood in terms of the corresponding continuous representatives.
\end{proposition}

For the Sobolev space comparisons, see \cite[Sections~2.3.2 and 2.5.6]{Tri83}; for the identification with H\"older spaces, see \cite[Section~2.5.7]{Tri83}; and for the critical continuity embeddings and their sharpness, see \cite[Section~2.7.1 and Chapter~9]{Tri83}. These results distinguish Sobolev embeddings from the arithmetic continuity estimate in Lemma~\ref{lem:besov-multiplier}.

\begin{proposition}
\label{appB:dyadic-step}
Let $1<p<\infty$ and $w\in L^p(\T)$. For $n\in\N_0$, let $w_n$ be the piecewise constant function that, on each interval $[k2^{-n},(k+1)2^{-n})$ ($0\le k<2^n$), equals the average of $w$, and, for $n\ge1$, write $d_n\defeq w_n-w_{n-1}$. If
\[
 \sum_{n\ge1}\|d_n\|_p^2<\infty,
\]
then $w\in B^0_{p,2}(\T)$ and
\begin{equation}\label{appB:eq:dyadic-Besov}
 \|w\|_{B^0_{p,2}}
 \le C_p\left(\|w\|_p+
        \Bigl(\sum_{n\ge1}\|d_n\|_p^2\Bigr)^{1/2}\right).
\end{equation}
\end{proposition}

This is the special case $s=0, q=2$ of the Haar frame characterization of Besov spaces in the range $1/p-1<s<1/p$; see \cite{GSU23}, pp.~2--3, equations~(1.5) and~(1.7), together with the periodic localization in \cite[Chapter~9]{Tri83}. Since $1<p<\infty$, the exponent $s=0$ lies strictly inside this range. The criterion characterizes the standard Besov space and is applied in Proposition~\ref{prop:conjugacy-regularity} to $w=h'$ in the conjugacy coordinate.

\section{Katznelson--Ornstein estimates}
\label{app:KO}

The following results are used in the proofs of Lemma~\ref{lem:KO-dyadic} and Theorem~\ref{thm:rigidity}, and in the comparison in Remark~\ref{rem:rigidity-methods}. The notation is fixed in Sections~\ref{subsec:conjugacy-regularity} and~\ref{subsec:global-linearization}.

\begin{proposition}
\label{appC:KO-square-summability}
Let $F\in\mathcal F_\alpha^2$, where $\alpha$ is irrational, and let $K_n^0$ be defined by~\eqref{eq:KO-return-distortion}. Then
\begin{equation}\label{eq:KO-square-summability}
 \sum_{n\ge1}(K_n^0)^2<\infty.
\end{equation}
\end{proposition}

This result is \cite[Theorem~3.1, p.~685; proof on pp.~686--689]{KO89}. Its assumptions that $\ln DF$ is absolutely continuous and that there exists $p>1$ such that $D\ln DF\in L^p$ hold automatically for every $F\in\mathcal F_\alpha^2$.

\begin{proposition}
\label{appC:KO-interval-ratios}
Let $F\in\mathcal F_\alpha^2$, assume that the partial quotients of $\alpha$ are uniformly bounded by $A$, and let $h$ satisfy~\eqref{eq:forward-conjugacy}. Choose $n_*\ge1$ such that $2^{1-n_*}<\Vert{}q_1\alpha\Vert{}$. For $n\ge n_*$, define
\[ \ell(n)\defeq \min\{m\ge1:\Vert{}q_m\alpha\Vert{}<2^{1-n}\}. \]
Then $\ell(n)\ge2$. There exist constants $C>0$ and $0<\theta<1$ such that, for every $n\ge n_*$ and $x\in\R$,
\begin{equation}\label{eq:appC-KO-interval-ratio}
 \left|\frac{h(x+2^{-n})-h(x)}{h(x)-h(x-2^{-n})}-1\right|
 \le C\sum_{m\ge\ell(n)}\theta^{m-\ell(n)}K_m^0.
\end{equation}
For each integer $m$, the conditions $\ell(n)=m$ and $n\ge n_*$ restrict the number of indices $n$ to at most $L_A\defeq3+\lfloor\log_2(A+1)\rfloor$. Moreover, $h$ and $h^{-1}$ are absolutely continuous and satisfy
\[ h',\ (h^{-1})'\in L^2(\T). \]
\end{proposition}

The interval estimate~\eqref{eq:appC-KO-interval-ratio} follows from the proof of \cite[Theorem~2.3, pp.~684--685, equations~2.2--2.3]{KO89} in the constant-type case; the finite multiplicity of $\ell(n)$ is given in \cite[p.~685]{KO89}. For absolute continuity and $L^2$ regularity of the derivatives, see \cite[Lemma~2.2, pp.~683--684; Theorem~2.3, p.~684; Theorem~3.1, p.~685]{KO89}.

\begin{proposition}
\label{appC:KO-smallness}
Let $F\in\mathcal F_\alpha^2$, where $\alpha$ is irrational, and let $K_n^0$ and $\widetilde K_n^1$ be given by~\eqref{eq:KO-return-distortion} and~\eqref{eq:KO-short-distortion}, respectively. For every $n\ge2$,
\begin{equation}\label{appC:KO-return-comparison}
 K_n^0\le2\widetilde K_n^1,
 \qquad \lim_{n\to\infty}\widetilde K_n^1=0,
\end{equation}
and, over all $q_n$-small intervals $I$,
\begin{equation}\label{appC:KO-return-smallness}
 \lim_{n\to\infty}\sup_I
 \int_I|D\ln DF^{q_n}(x)|\,dx=0.
\end{equation}
\end{proposition}

The inequality $K_n^0\le2\widetilde K_n^1$ is given in \cite[Lemma~1.2, p.~682, equation~1.1]{KO89}. For the derivation of the two limits under the $C^2$ assumption, see \cite[Lemmas~1.7 and 1.9, p.~648; Lemma~3.1, p.~652; Proposition~3.3, p.~653; Corollary~3.4, p.~655]{KO89D}.

\begin{proposition}
\label{appC:KO-recursion}
Let $F\in\mathcal F_\alpha^2$, where $\alpha$ is irrational, and let $\mathsf G_j$ and $\mathsf g_j$ be given by~\eqref{eq:dynamical-average} and~\eqref{eq:KO-averaging}, respectively. Fix $1<p<2$. There exist constants $C>0$ and $0<\eta\le\eta_1<1$, depending only on $F$ and $p$, such that, for every sufficiently small $\varepsilon>0$, there are an index $n_0\ge1$ and nonnegative sequences $(\mathsf r_n)_{n\ge n_0}$ and $(c_{nj})_{n\ge n_0,j\ge1}$ satisfying, for $n\ge n_0$,
\begin{equation}\label{appC:KO-recursion-estimate}
 \widetilde K_n^1
 \le C\varepsilon
       \sum_{n/2<m<n}\eta^{n-m}\widetilde K_m^1
      +\mathsf r_n+C\sum_{j\ge1}c_{nj}\|\mathsf g_j\|_p,
\end{equation}
where
\begin{equation}\label{appC:KO-recursion-remainder}
 \sum_{n\ge n_0}\mathsf r_n<\infty,
 \qquad
 0\le c_{nj}\le
 \begin{cases}
 n\eta^{n/2}, & j<n/2, \\
 \eta_1^{n-j}, & n/2\le j\le n, \\
 \eta_1^{j-n}, & j>n.
 \end{cases}
\end{equation}
The constants $C,\eta,\eta_1$ are independent of $\varepsilon$, whereas the index $n_0$ and the sum $\sum_{n\ge n_0}\mathsf r_n$ depend on $\varepsilon$.
\end{proposition}

The recurrence estimates~\eqref{appC:KO-recursion-estimate} and~\eqref{appC:KO-recursion-remainder} are given in \cite[proof of Theorem~3.1, pp.~688--689, equations~3.19--3.22]{KO89}. By \cite[p.~683, equation~1.8]{KO89}, the terms with $m\le n/2$ give a summable contribution, which is included in $\mathsf r_n$ together with the exponentially decaying term in equation~3.19. Combining~\eqref{appC:KO-return-smallness} with a sufficiently large fixed return order in \cite[p.~688, equations~3.17--3.19]{KO89} gives the small factor $C\varepsilon$. These estimates do not require a priori summability of $(\|\mathsf g_j\|_p)_{j\ge1}$.

\begin{proposition}
\label{appC:KO-C1-criterion}
Let $F\in\mathcal F_\alpha^2$, let $\alpha$ be an irrational number of constant type, and let $h$ satisfy~\eqref{eq:forward-conjugacy}. If $\sum_{n\ge1}K_n^0<\infty$, then $h\in \mathcal D^1(\T)$. If $F''\in C^\sigma(\T)$ (where $0<\sigma<1$), then there exist constants $C>0$ and $0<\lambda<1$ such that, for every $n\ge1$,
\begin{equation}\label{eq:KO-exponential-decay}
 K_n^0\le C\lambda^n.
\end{equation}
In particular, this summability condition holds for every $F\in\mathcal F_\alpha^{2+\sigma}$.
\end{proposition}

The first assertion is the conclusion of \cite[Theorem~2.4, pp.~650--651]{KO89D} under the constant-type assumption: bounded partial quotients make the weighted summability assumption equivalent to $\sum_{n\ge1}K_n^0<\infty$. The exponential decay estimate~\eqref{eq:KO-exponential-decay} is given in \cite[Theorem~3.7, p.~655]{KO89D} and does not depend on the constant-type assumption.

\par\medskip\noindent\textbf{Data availability.}
No datasets were generated or analyzed during the current study.
\par\medskip\noindent\textbf{Conflict of interest.}
The authors declare no conflict of interest.
\par\medskip\noindent\textbf{AI declaration.}
The authors used OpenAI's ChatGPT as an auxiliary tool for proofreading and improving the presentation of some arguments.
All mathematical arguments were independently verified by the authors, who take full
responsibility for the content.

\end{document}